\documentclass{article}

\usepackage[left=3cm, right=3cm, top=2cm, bottom=2.5cm]{geometry}
\usepackage[colorlinks]{hyperref}
\hypersetup{linkcolor=blue, urlcolor=blue, citecolor=red}
\usepackage[bf,singlelinecheck=false]{caption}
\usepackage{amsmath, amssymb, amsthm, enumerate, tikz, cite, caption, graphicx,
mathrsfs, color, comment, xcolor, algorithmicx,  mathtools, enumitem}
\usepackage[normalem]{ulem}

\usetikzlibrary{arrows, shapes, patterns, decorations.pathmorphing}
\newcommand{\arc}[2]{\path{} (#1) edge [->,thick] node {} (#2);}

\newcommand{\arcSym}[2]{\path{} (#1) edge [<->,thick] node {} (#2);}
\newcommand{\arcDash}[2]{\path{} (#1) edge [->,dashed,thick] node {} (#2);}

\newcommand{\arcDot}[2]{\path{} (#1) edge [->,dotted,thick] node {} (#2);}
\newcommand{\looparcL}[1]{\path{} (#1) edge [loop left,thick] node {} (#1);}
\newcommand{\looparcR}[1]{\path{} (#1) edge [loop right,thick] node {} (#1);}

\def\<{\langle}
\def\>{\rangle}
\newcommand{\Stacks}[1]{\operatorname{Stacks}(#1)}
\newcommand{\stacks}{\operatorname{Stacks}(\mathcal{O})}

\newcommand{\Sn}[1]{\mathcal{S}_{#1}}
\newcommand{\Cn}[1]{\mathcal{C}_{#1}}
\newcommand{\Sym}[1]{\operatorname{Sym}(#1)}

\newcommand{\refine}[2]{{#2}\mathop{\Vert}{#1}(#2)}
\newcommand{\Auto}[1]{\operatorname{Stab}(#1)}
\newcommand{\Iso}[2]{\operatorname{Transp}(#1,\,#2)}

\newcommand{\SOm}{\Sym{\Omega}}

\newcommand{\ov}[1]{\overline{#1}}

\newcommand{\N}{\mathbb{N}}

\newcommand{\set}[2]{\{#1:#2\}}
\newcommand{\bigset}[2]{\left\{#1:#2\right\}}

\renewcommand{\to}{\rightarrow}

\newtheorem{thm}{Theorem}[section]
\newtheorem{lemma}[thm]{Lemma}
\newtheorem{cor}[thm]{Corollary}
\newtheorem{prop}[thm]{Proposition}
\newtheorem{remark}[thm]{Remark}

\theoremstyle{definition}
\newtheorem{example}[thm]{Example}
\newtheorem{defn}[thm]{Definition}
\newtheorem{notation}[thm]{Notation}

\title{
  Perfect refiners for permutation group backtracking algorithms
}
\author{
  Christopher Jefferson,
  Rebecca Waldecker,
  Wilf A.\,Wilson
}

\begin{document}

\maketitle

\begin{abstract}
  Backtrack search is a fundamental technique for computing with finite permutation groups, which has been formulated in terms of points, ordered partitions, and graphs. We provide a framework for discussing the most common forms of backtrack search in a generic way.
  We introduce the concept of perfect refiners to better understand and compare the pruning power available in these different settings.
  We also present a new formulation of backtrack search, which allows the use of graphs with additional vertices, and which is implemented in the software package \textsc{Vole}.
  For each setting,
  we classify the groups and cosets for which there exist perfect refiners.  
  Moreover, we describe perfect refiners for many naturally-occurring examples of stabilisers and transporter sets, including applications to normaliser
  and subgroup conjugacy problems for 2-closed groups.
\end{abstract}


\section{Introduction}\label{sec-intro}

A careful backtrack search through the elements of a symmetric group is the fastest general purpose technique, in practice, for solving many computational problems in finite permutation groups.
This includes finding normalisers and intersections of subgroups, and
stabilisers of sets and graphs.

At the heart of these backtracking algorithms are refiners.
Refiners are functions which are used to prune redundant parts of the search.
They provide the main facility for taking into account the structural aspects of the problem at hand, and making clever deductions. Better refiners are more likely to prune more efficiently, and can reduce search times by orders of magnitude.
Refiners, therefore, are an obvious and ongoing target of further research (see
for example~\cite{newrefiners} for recent work in this area), and they are the focus of this article.

Roughly speaking, backtracking algorithms organise permutations as bijective maps 
on sets of combinatorial structures. The first version was originally formulated by Sims,
organised around a base and strong generating set (see~\cite{Sims1971}).
Two decades later, Leon reformulated this in terms of ordered partitions~\cite{Leon1991,Leon1997}, and thereby obtained significantly improved performance in many cases.
This technique was enhanced with orbital graphs in~\cite{Theissen97,newrefiners}.
More recently, this use of orbital graphs inspired a reformulation in~\cite{GB_published}
around labelled digraphs.

The motivation for organising a search around more sophisticated structures is to allow some sets of permutations to be represented with greater fidelity,
so that refiners can prune more effectively.
However, when choosing the appropriate search infrastructure for a particular problem, we need to keep in mind that more complicated structures do not always offer better pruning, that they are more expensive to compute with, and that it is easy to find relatively small cases where \textit{all} existing backtracking algorithms exhibit poor performance. This holds especially for subgroup conjugacy and normaliser problems.
It is therefore important to develop tools for understanding which problems are well suited to which kinds of backtrack search and refiners, and to understand the limitations of the existing techniques, with the aim of developing improvements.

This article has three main purposes.
Firstly, we introduce the concept of {perfect} refiners, which are those with maximal pruning power within a given search framework. This concept is also meant to initiate a discussion of the quality of refiners, both within and across various backtrack search techniques.
Secondly, we describe an extension to the graph backtracking technique that permits the use of graphs with additional vertices, and which enables many more perfect refiners than previous techniques. There already is an implementation available, in \textsc{Vole}~\cite{vole}.
Finally, we partially classify, and give examples of, perfect refiners in each setting of backtrack search.

Our results show that graph backtracking and extended graph backtracking admit perfect refiners for many natural problems. This goes some way to explaining the experimental data in~\cite[Section~9]{GB_published}, which demonstrated that many problems could be solved without actual backtracking happening during the search.
Furthermore, we find that extended backtracking enables improved refiners for normaliser and conjugacy problems.
In particular, we give refiners for normalisers and conjugacy
within the extended graph backtracking framework that are perfect in some cases.

This article is organised as follows.
In Section~\ref{sec-prelim}, we present some necessary background definitions and notation;
in particular, we briefly describe backtrack search in finite symmetric groups.
We then give definitions and results about arbitrary refiners in Section~\ref{sec-refiners} and perfect refiners in Section~\ref{sec-theory-perfect}.
In Section~\ref{sec-perfect-stab-transp}, we examine perfect refiners for stabiliser and transporter problems.
In Section~\ref{sec-extended}, we introduce extended graph backtracking.
In Section~\ref{sec-hierarchy},
we compare the various kinds of backtrack search and their potentials for perfect refinement.
In Section~\ref{sec-examples-perfect},
we give examples of perfect refiners for many natural problems.
We conclude with some final remarks in Section~\ref{sec-outro}.

\subsection*{Acknowledgments}\label{sec-acknowledgments}

The results discussed here build on work in~\cite{GB_published}, and we include elements of an earlier draft of that article~\cite[Section~5.1]{GB_extended}.
We therefore thank the VolkswagenStiftung (\textbf{Grant no.~93764})
and the Royal Society (\textbf{Grant code URF\textbackslash R\textbackslash 180015})
again for their financial support of this earlier work.
For financial support during the more recent advances,
we thank the DFG (\textbf{Grant no.~WA 3089/9-1}) and again the Royal Society
(\textbf{Grant codes} \textbf{RGF\textbackslash EA\textbackslash 181005} and \textbf{URF\textbackslash R\textbackslash 180015}).

\section{Background and notation}\label{sec-prelim}

In this section, we give some notation and definitions, and we introduce background concepts including 
backtrack search in the symmetric group.
Our terminology and notation closely follows that of~\cite{GB_published}.

Sets and lists feature prominently in this article: a set is an unordered duplicate-free collection of objects, whereas a list has an ordering, and may contain duplicates.
If \(L\) and \(K\) are lists, then \(|L|\) denotes the number of elements in \(L\), \(L \Vert K\) denotes the concatenation of \(L\) and \(K\),
and if \(i \in \{1,\ldots,|L|\}\), then \(L[i]\) denotes the element of \(L\) in the \(i\)-th position.

Throughout this article, \(\Omega\) is a nonempty finite
set, and \(\SOm\) is the symmetric group on \(\Omega\).
We denote any action of \(\SOm\) by exponentiation.
Given an action of \(\SOm\) on a set \(\mathcal{O}\), we iteratively define an action of \(\SOm\) induced on the set of all subsets of \(\mathcal{O}\), and on the set of all finite lists with elements in \(\mathcal{O}\).
In more detail, for all \(x_1, \ldots, x_k \in \mathcal{O}\) and \(g \in \SOm\),
we define \(\{x_1,\ldots,x_k\}^{g} \coloneqq \{x_1^{g},\ldots,x_k^{g}\}\)
and \([x_1,\ldots,x_k]^{g} \coloneqq [x_1^{g},\ldots,x_k^{g}]\).
For example, if \(\Omega \coloneqq \{1,2,3,4\}\), then the image of the set-of-lists
\(\{[1,2],[2,3],[3,2]\}\) under the permutation \((1\,2)(3\,4)\) is
\[\{[1,2],[2,3],[3,2]\}^{(1\,2)(3\,4)}
= \{[1,2]^{(1\,2)(3\,4)},[2,3]^{(1\,2)(3\,4)},[3,2]^{(1\,2)(3\,4)}\}
= \{[2,1],[1,4],[4,1]\}.\]

Let \(\mathcal{O}\) be a set on which \(\SOm\) acts.
We define \(\stacks\) to be the set of all finite lists with entries in \(\mathcal{O}\) 
(we call these lists \emph{stacks}), and \(\stacks^{\stacks}\) to be the set of all functions from \(\stacks\) to itself.
If \(x,y \in \mathcal{O}\),
then the \emph{transporter set} from \(x\) to \(y\) in \(\SOm\) is denoted by \(\Iso{x}{y} \coloneqq \set{g \in \SOm}{x^{g} = y}\), and the \emph{stabiliser} of \(x\) in \(\SOm\) is the subgroup \(\Auto{x} \coloneqq \Iso{x}{x}\) of all permutations in \(\SOm\) that fix \(x\) under the action.
Note that
for all \(h \in \SOm\), \(\Iso{x}{y^{h}} = \Iso{x}{y} \cdot h\).

The \emph{setwise stabiliser} of a collection \(x_1, \ldots, x_n\) of objects is the stabiliser of the set \(\{x_1, \ldots, x_n\}\),
whereas the \emph{pointwise stabiliser} is \(\bigcap_{i = 1}^{n} \Auto{x_i}\), which is equal to the stabiliser of any list of the objects, such as \([x_1, \ldots, x_n]\).

The backtrack search techniques considered in this article are organised around lists of points, ordered partitions, and graphs.
An \emph{ordered partition} of a set \(V\) is a list of nonempty disjoint subsets of \(V\)
whose union is \(V\).
A \emph{graph} on a set \(V\) is a pair \((V, E)\) of \emph{vertices} \(V\) and a set of \(2\)-subsets of \(V\) called \emph{edges}.
Similarly, a \emph{digraph} on \(V\) is a pair \((V, A)\) of \emph{vertices} \(V\) and a set of ordered pairs in \(V\) called \emph{arcs}.
The action of \(\Sym{V}\) on the sets of all graphs and digraphs on \(V\) is defined, respectively, by \((V,E)^{g} \coloneqq (V,E^{g})\) and \((V,A)^{g} \coloneqq (V,A^{g})\) for all \(g \in \Sym{V}\), edge sets \(E\), and arc sets \(A\).
A \emph{labelled digraph} is a digraph with an assignment of a label to each
vertex and arc. See~\cite[Section~2.1]{GB_published} for more information.
Whenever, in this article, we write that `\((V, E)\) is a graph', then this means that \(V\) is the set of vertices and \(E\)
is the set of edges as explained above. In the same way, we use the remaining notation introduced here without further explanation.

\subsection{Classical backtracking, partition backtracking, and graph backtracking}\label{sec-backtrack-types}

We briefly summarise the concept of backtrack search in \(\SOm\), and introduce some terminology.

Let \(U_1, \ldots, U_k\) be subsets of \(\SOm\) for some \(k \in \N\), and suppose that we wish to search for the intersection \(U_1 \cap \cdots \cap U_k\).
For this technique to be useful in practice, it should be computationally cheap, for each \(i \in \{1,\ldots,k\}\), 
to determine whether any given element of \(\SOm\) is contained in \(U_{i}\).

Many typical search problems can be formulated in this way.
For example, if \(U_1\) is a subgroup of \(\SOm\) given by generators and \(U_2 \coloneqq \Iso{\Gamma}{\Delta}\) for some graphs \(\Gamma\) and \(\Delta\) with vertex set \(\Omega\),
then searching for an element of \(U_1 \cap U_2\) solves the graph isomorphism problem for \(\Gamma\) and \(\Delta\) in \(U_1\).

Let \(\mathcal{O}\) be a set on which \(\SOm\) acts (such as \(\Omega\) itself, or the set of all ordered partitions of \(\Omega\)).
In this paper we will present all search techniques in a common framework. The fundamental idea of this framework is to organise a backtrack search for \(U_1 \cap \cdots \cap U_k\) around a pair of stacks of objects in \(\mathcal{O}\).
At any point in the algorithm, when the pair of stacks is \((S,T)\) for some \(S,T \in \stacks\), the set of permutations being searched is \(\Iso{S}{T}\).
The stacks are initially empty, which means that the search begins with the whole symmetric group.
In each step down into the recursion, new stacks are appended to the existing ones. This may be done by a refiner, in an attempt to remove redundant parts of the search space (Section~\ref{sec-refiners}), or by a splitter, in order to divide the search space into smaller parts that can be searched recursively (see~\cite[Section~6]{GB_published}).

For such a backtrack search to be practical, the set \(\mathcal{O}\) should be easy to compute with, and in particular, it should be relatively cheap to compute stabilisers and transporter sets in \(\SOm\) of elements in \(\mathcal{O}\), or at least to obtain close overapproximations of them (see~\cite[Section~5]{GB_published}).
On the other hand, the set \(\mathcal{O}\) should be sufficiently rich and varied that refiners can construct stacks in \(\mathcal{O}\) that encode useful information about the problem at hand.

So far, backtrack search in finite symmetric groups has been formulated and implemented in several settings.
We use the term \emph{classical backtracking} for the case that \(\mathcal{O} = \Omega\); this is essentially the original backtrack search in \(\SOm\) introduced in~\cite{Sims1971}, although presented differently.
\emph{Partition backtracking} is backtrack search where the objects are ordered partitions of \(\Omega\); this is essentially the technique introduced in~\cite{Leon1991}.
We use the term \emph{graph backtracking} when the objects are labelled digraphs on \(\Omega\)~\cite{GB_published}.
In Section~\ref{sec-extended}, we introduce an advancement of this latter technique, which we call \emph{extended graph backtracking}.

\section{Refiners for backtrack search}\label{sec-refiners}

Throughout this section and Sections~\ref{sec-theory-perfect} and~\ref{sec-perfect-stab-transp}, we let \(\mathcal{O}\) be a set on which \(\SOm\) acts, so that the setting is backtrack search organised around stacks in \(\mathcal{O}\). In particular, all definitions and results are relative to the set \(\mathcal{O}\), even if we do not explicitly repeat that every single time. 

\begin{defn}\label{defn-refiner}
  A \emph{refiner} for a set of permutations \(U \subseteq \SOm\) is a pair \((f_L,f_R)\) of functions in \(\stacks^{\stacks}\) such that:
  \begin{equation}\label{refiner-condition}
    \text{For all}\ S,T \in \stacks,\ U \cap \Iso{S}{T} \subseteq \Iso{f_L(S)}{f_R(T)}.
    \tag{\(\ast\)}
  \end{equation}
\end{defn}


In practice, we focus on refiners for subsets of \(\SOm\) that are subgroups, cosets of subgroups, or empty.

We remark that any pair of functions in \(\stacks^{\stacks}\) is a refiner for the empty set and that it is necessary to include the empty set
as a possibility because, for example, it is common to search for transporter sets, which may be empty.

Let \((f_L,f_R)\) be a refiner for a set \(U\subseteq\SOm\),
let \(S, T \in \stacks\) with \(|S|=|T|\),
and suppose that we are part way through a search for an intersection of subsets of \(\SOm\) that includes \(U\),
with \(\Iso{S}{T}\) as the current search space.
In this situation, the aim of applying this refiner to the stacks \(S\) and
\(T\) is to prune elements of \(\SOm \setminus U\) from the current search space, by moving to the transporter set corresponding to the lengthened stacks \(\refine{f_L}{S}\) and \(\refine{f_R}{T}\).
By~\eqref{refiner-condition}, \(\Iso{\refine{f_L}{S}}{\refine{f_R}{T}}\) retains the elements of \(\Iso{S}{T}\) that \emph{are} in \(U\),
but on the other hand, \(\Iso{\refine{f_L}{S}}{\refine{f_R}{T}}\) is contained in \(\Iso{S}{T}\), perhaps properly, and may therefore lack some elements of \(\Iso{S}{T}\) that are \emph{not} in \(U\).
In particular, if \(\Iso{\refine{f_L}{S}}{\refine{f_R}{T}}\) is empty, then the search can backtrack.

Leon used the term \(\mathcal{P}\)\emph{-refinement} in his work on partition backtracking~\cite{Leon1991,Leon1997} for a similar concept
that is essentially compatible with our notion of a refiner,
although he presents it in a significantly different fashion.
Definition~\ref{defn-refiner} matches the notion of a refiner used in~\cite[Section~5]{GB_published}.

We remark that the definition of a refiner guarantees no particular success at pruning.
For example, if \(\iota\) is the identity map on \(\stacks\), and if \(\varepsilon\) is the constant map on \(\stacks\) whose image is the empty stack in \(\stacks\),
then \((\iota,\iota)\) and \((\varepsilon,\varepsilon)\) are refiners for every subset of \(\SOm\), even though neither performs any pruning.  Thus it is desirable to have a measure of the quality of a refiner.
In the following section we introduce perfect refiners, which are those that have maximal pruning power.

To simplify some forthcoming exposition, we introduce a way of applying permutations to functions in \(\stacks^{\stacks}\).
It is straightforward to verify that this defines an action of \(\SOm\).

\begin{notation}\label{notation-sym-on-refiners}
  For any \(f \in \stacks^{\stacks}\) and \(x \in \SOm\), we define  \(f^{x}\) as follows:\\
  \(f^{x}(S) \coloneqq f(S^{x^{-1}})^{x}\) for all \(S \in \stacks\).
\end{notation}

\begin{lemma}\label{lem-sym-acts-on-refiners}
  For all \(f \in \stacks^{\stacks}\) and \(x \in \SOm\), the map \(f^{x}\) explained above is in 
   \(\stacks^{\stacks}\).
   Moreover, \(\SOm\) acts on \(\stacks^{\stacks}\). 
\end{lemma}


The following results show the close relationships between the two functions that comprise a refiner for a subgroup (Lemma~\ref{lem-simple-refiner}) or more generally for a coset of a subgroup (Lemma~\ref{lem-refiner-group-coset}).
We note that Lemma~\ref{lem-refiner-group-coset} is a reformulation of~\cite[Lemma~4.6]{GB_published} that uses Notation~\ref{notation-sym-on-refiners}.

\begin{lemma}[\mbox{Lemma~4.4 in~\cite{GB_published}; cf.~\cite[Lemma~6]{Leon1991} and~\cite[Prop 2]{Leon1997}}]\label{lem-simple-refiner}
  If \((f_L,f_R)\) is a refiner for a subgroup of \(\SOm\), then \(f_L = f_R\).
\end{lemma}

\begin{proof}
  Suppose that \((f_L,f_R)\) is a refiner for a group \(G\) and let \(S\in\stacks\).
  Since \(1_G \in \Iso{S}{S}\), it follows by~\eqref{refiner-condition} that \(1_G \in \Iso{f_L(S)}{f_R(S)}\), which means that \(f_L(S) = f_L(S)^{1_G} = f_R(S)\).
\end{proof}

\begin{lemma}\label{lem-refiner-group-coset}
  Let \(G \leq \SOm\), \(x \in \SOm\), and \(f_L,f_R \in \stacks^{\stacks}\).
  Then \((f_L,f_R)\) is a refiner for the right coset \(G x\) if and only if \((f_L,f_L)\) is a refiner for \(G\) and \(f_R = {f_L}^{x}\).
\end{lemma}

By Lemmas~\ref{lem-simple-refiner} and~\ref{lem-refiner-group-coset}, a refiner
for a group or coset is built from a single function from \(\stacks\) to itself. The following lemma gives equivalent conditions for such a function to yield a refiner (always with respect to the set \(\mathcal{O}\), as mentioned earlier):

\begin{lemma}
  Let \(G \leq \SOm\) and \(f \in \stacks^{\stacks}\).
  The following are equivalent:
  \begin{enumerate}[label=\emph{(\roman*)},ref=\textrm{(\roman*)}]
    \item\label{item-sym}
      \((f, f)\) is a refiner for \(G\).
    \item\label{item-invariant}
      \(f(S^{x}) = f(S)^{x}\) for all \(S \in \stacks\) and \(x \in G\).
    \item\label{item-G-in-stab}
      \(G \leq \Auto{f}\).
  \end{enumerate}
\end{lemma}

\begin{proof}
  Suppose that~\ref{item-sym} holds and let \(S \in \stacks\) and \(x \in G\).
  By Definition~\ref{defn-refiner}, \(G \cap \Iso{S}{S^{x}} \subseteq \Iso{f(S)}{f(S^{x})}\).
  Then~\ref{item-invariant} holds, since \(x \in G \cap \Iso{S}{S^{x}}\).
  Conversely, suppose that~\ref{item-invariant} holds and let \(S,T \in \stacks\).
  If \(x \in G \cap \Iso{S}{T}\), then \(T = S^{x}\), and since \(f(S)^{x} = f(S^{x})\) by assumption, it follows that \(x \in \Iso{f(S)}{f(T)}\). Therefore~\ref{item-sym} holds by Definition~\ref{defn-refiner}.

  The
  equivalence of~\ref{item-invariant} and~\ref{item-G-in-stab} is clear from Notation~\ref{notation-sym-on-refiners} and the definition of \(\Auto{f}\).
\end{proof}

Next, we show a way for refiners to be combined to give a refiner for an intersection of sets.

\begin{notation}\label{notation-refiner-combine}
  For all \(f,g \in \stacks^{\stacks}\),
  we define the function \(f \mathop{\Vert} g \in \stacks^{\stacks}\) by
  \((f \mathop{\Vert} g)(S) \coloneqq f(S) \mathop{\Vert} g(S)\) for all \(S \in \stacks\).
\end{notation}

\begin{lemma}\label{lem-refiner-intersection}
  Let \((f, \sigma)\) be a refiner for the set \(U \subseteq \SOm\),
  and let \((g, \tau)\) be a refiner for \(V \subseteq \SOm\).
  Then \((f \mathop{\Vert} g, \sigma \mathop{\Vert} \tau)\) is a refiner for \(U \cap V\).
\end{lemma}

\begin{proof}
  We show that the set \(U \cap V\) and the pair of functions \((f \mathop{\Vert} g, \sigma \mathop{\Vert} \tau)\) satisfy~\eqref{refiner-condition}.
  Let \(S,T \in \stacks\).
  If \(|f(S)| \neq |\sigma(T)|\), then \(\Iso{f(S)}{\sigma(T)} = \varnothing\), and so \(U \cap \Iso{S}{T} = \varnothing\), since \((f,\sigma)\) is a refiner for \(U\).
  In particular, \((U \cap V) \cap \Iso{S}{T} = \varnothing\).
  If instead \(|f(S)| = |\sigma(T)|\), then
  \begin{align*}
    (U \cap V) \cap \Iso{S}{T}
      &         = (U \cap \Iso{S}{T}) \cap (V \cap \Iso{S}{T}) \\
      & \subseteq \Iso{f(S)}{\sigma(T)}) \cap \Iso{g(S)}{\tau(T)} \\
      &         = \Iso{f(S) \Vert g(S)}{\sigma(T) \mathop{\Vert} \tau(T)} \\
      &         = \Iso{(f \mathop{\Vert} g)(S)}{(\sigma \mathop{\Vert} \tau)(T)}.
      \qedhere
  \end{align*}
\end{proof}

We remark that the \(\Vert\) operation is associative, and that we may repeatedly apply Lemma~\ref{lem-refiner-intersection} to obtain a refiner for any finite intersection of sets, given a refiner for each set.

\section{Perfect refiners}\label{sec-theory-perfect}

\begin{lemma}\label{lem-perfect-is-refiner}
  Let \(U \subseteq \SOm\),
  let \(f_{L},f_{R} \in \stacks^{\stacks}\), and suppose that
  \begin{equation}\label{perfect-condition}
    U \cap \Iso{S}{T} = \Iso{\refine{f_L}{S}}{\refine{f_R}{T}}
    \tag{\(\circledast\)}
  \end{equation}
  for all \(S, T \in \stacks\) with \(|S| = |T|\).
  Then \((f_{L}, f_{R})\) is a refiner for \(U\).
\end{lemma}

\begin{proof}
  To show that~\eqref{refiner-condition} holds, let \(S,T\in\stacks\).
  If \(\Iso{S}{T} = \varnothing\), then this is clear, so suppose otherwise; in particular, suppose that \(|S| = |T|\). Then
  \begin{align*}
    U \cap \Iso{S}{T} & = \Iso{\refine{f_L}{S}}{\refine{f_R}{T}} \\
                      & = \Iso{S}{T} \cap \Iso{f_L(S)}{f_R(T)} \\
                      & \subseteq \Iso{f_L(S)}{f_R(T)}. \qedhere
  \end{align*}
\end{proof}

\begin{defn}\label{defn-perfect}
  Refiners with the property~\eqref{perfect-condition} from Lemma~\ref{lem-perfect-is-refiner} are called \emph{perfect refiners} (with respect to \(\mathcal{O}\), which we will usually omit).
\end{defn}

We give several examples of perfect refiners in Section~\ref{sec-examples-perfect}.
Note that a refiner for a set \(U\subseteq\SOm\) is also a refiner for any proper subset of \(U\), but never a perfect refiner.
The equation `\(|S| = |T|\)' in~\eqref{perfect-condition} is included for
convenience:
in principle, two stacks of different lengths (whose transporter set is empty) could be extended to stacks with nonempty transporter set.
However, a refiner is only ever applied to pairs of stacks of equal lengths, so we remove the need to deal with this complication.

During a search, a perfect refiner for \(U\) can be used to extend the stacks \(S\) and \(T\) (representing the current search space \(\Iso{S}{T}\)) to obtain the potentially-smaller search space \(\Iso{S}{T} \cap U\).
Thus a perfect refiner for \(U\) allows \(U\) to be represented with full fidelity,
and no effort must be wasted considering elements of \(\SOm\) that are not in \(U\). In other words:
Perfect refiners are the refiners with maximal pruning power.

A perfect refiner therefore needs to be applied at most once in any branch of search, and should therefore be applied at the root node to avoid unnecessary repetition.
It is for this reason that perfect refiners often comprise constant functions in practice; see Lemma~\ref{lem-refiner-constant}.

Note that if each set \(U_i\) in a search for \(U_1 \cap \cdots \cap U_n\) is given with a perfect refiner, then the search can terminate at the root node without splitting or backtracking.
This is because the search begins with \(S\) and \(T\) being empty stacks, and so their transporter set is initially \(\SOm\).  Applying the perfect refiners in turn then gives stacks with transporter set \(U_1 \cap \cdots \cap U_n\), as required.

In the following lemmas, we see that constructing a perfect refiner for a set \(U\) is equivalent to finding a pair of stacks in \(\mathcal{O}\) with transporter set \(U\).
It follows that the existence of a perfect refiner for any given subset of
\(\SOm\), in backtrack search organised around stacks in \(\mathcal{O}\), depends on the choice of \(\mathcal{O}\). But as before, we will not always add ``with respect to \(\mathcal{O}\)''.

\begin{lemma}\label{lem-refiner-constant}
  Let \(U \subseteq \SOm\), and suppose there exist \(A,B \in \stacks\) such that \(U \subseteq \Iso{A}{B}\).
  Let \(f_A\) and \(f_B\) be constant functions on \(\stacks\) with images \(A\) and \(B\), respectively.
  Then \((f_A, f_B)\) is a refiner for \(U\).
  Moreover, if \(U = \Iso{A}{B}\), then this refiner is perfect.
\end{lemma}

\begin{proof}
  The pair of functions satisfy the condition in Definition~\ref{defn-refiner}, since for all \(S,T\in\stacks\):
  \[
    \Iso{S}{T} \cap U \subseteq U \subseteq \Iso{A}{B} = \Iso{f_A(S)}{f_B(T)}.
  \]
  If \(U = \Iso{A}{B}\), then for all \(S,T\in\stacks\) with \(|S|=|T|\), condition~\eqref{perfect-condition} is satisfied, since
  \begin{align*}
    \Iso{S}{T} \cap U & = \Iso{S}{T} \cap \Iso{A}{B} \\
                      & = \Iso{S}{T} \cap \Iso{f_A(S)}{f_B(T)} \\
                      & = \Iso{\refine{f_A}{S}}{\refine{f_B}{T}}.
    \qedhere
  \end{align*}
\end{proof}

\begin{cor}\label{cor-group-refiner-constant}
  Let \(G \leq \SOm\), and suppose there exists \(A \in \stacks\) such that \(G \leq \Auto{A}\).
  Define \(f_A\) to be the constant function on \(\stacks\) with image \(A\).
  Then \((f_A,f_A)\) is a refiner for \(G\).
  Moreover, if \(G = \Auto{A}\), then this refiner is perfect.
\end{cor}

\begin{lemma}\label{lem-when-perfect-refiners-exist}
  Let \(U \subseteq \SOm\).
  There exists a perfect refiner for \(U\) if and only if \(U = \Iso{S}{T}\) for some \(S, T \in \stacks\).
  In particular, there exists a perfect refiner for a group \(G \leq \SOm\) if and only if \(G = \Auto{S}\) for some \(S\in\stacks\).
\end{lemma}

\begin{proof}
  Let \((f_{L}, f_{R})\) be a perfect refiner for \(U\) and let \(S\) be the empty stack in \(\stacks\).
  Note that \(\Iso{S}{S} = \SOm\). Then \(U = \Iso{f_L(S)}{f_R(S)}\) by Definition~\ref{defn-perfect}.
  The converse implication follows from Lemma~\ref{lem-refiner-constant}.
\end{proof}

\begin{cor}\label{cor-when-perfect-refiners-exist}
  If there exists a perfect refiner for a subset \(U \subseteq \SOm\),
  then either \(U\) is empty, or \(U\) is a subgroup, or a coset of a subgroup, of \(\SOm\).
\end{cor}

\begin{lemma}[cf.\ Lemma~\ref{lem-refiner-group-coset}]\label{lem-perfect-refiner-group-coset}
  Let \(G \leq \SOm\), let \(x \in \SOm\), and let \(f_L,f_R \in \stacks^{\stacks}\).
  Then \((f_L,f_R)\) is a perfect refiner for \(G x\) if and only if \((f_L,f_L)\) is a perfect refiner for \(G\) and \(f_R = f_L^{x}\).
  In particular, there exists a perfect refiner for a group \(G \leq \SOm\) if and only if there exist perfect refiners for one, and hence all, cosets of \(G\) in \(\SOm\).
\end{lemma}

\begin{proof}
  We prove that for all \(x,y\in\SOm\), \((f_L,f_{L}^{x})\) is a perfect refiner for \(Gx\) if and only if \((f_L,f_{L}^{y})\) is a perfect refiner for \(Gy\). The lemma follows from this by choosing \(y\coloneqq 1_{G}\) and by Lemma~\ref{lem-refiner-group-coset}.

  Let \(S,T \in \stacks\) with \(|S| = |T|\). Suppose that \((f_L,f_{L}^{x})\) is a perfect refiner for \(Gx\).
  Note that \(G y \cap \Iso{S}{T} = (G x \cap \Iso{S}{T^{y^{-1} x}}) \cdot x^{-1} y\).
  Then by assumption:
  \begin{align*}
    Gx \cap \Iso{S}{T^{y^{-1} x}}
      & = \Iso{\refine{f_L}{S}}{\refine{f_{L}^{x}}{T^{y^{-1} x}}} \\
      & = \Iso{\refine{f_L}{S}}{T^{y^{-1} x} \Vert{f_{L}^{y}}(T)^{y^{-1} x}} \\
      & = \Iso{\refine{f_L}{S}}{(\refine{f_{L}^{y}}{T})^{y^{-1} x}} \\
      & = \Iso{\refine{f_L}{S}}{\refine{f_{L}^{y}}{T}} \cdot y^{-1} x.
  \end{align*}
  Hence \((f_L,f_L^{y})\) is a perfect refiner for \(Gy\) by~\eqref{perfect-condition}.
  The converse implication follows by symmetry.
\end{proof}

The following lemma shows that combining perfect refiners with the \(\Vert\) operation of Notation~\ref{notation-refiner-combine} preserves perfectness.
This implies that the collection of groups and cosets which have perfect refiners is closed under intersection.

\begin{lemma}\label{lem-perfect-refiner-intersection}
  Let \((f, \sigma)\) and \((g, \tau)\) be perfect refiners for the sets \(U, V \subseteq \SOm\), respectively.
  Then \((f \Vert g, \sigma \Vert \tau)\) is a perfect refiner for the set \(U \cap V\).
\end{lemma}

\begin{proof}
  We show that condition~\eqref{perfect-condition} from Lemma~\ref{lem-perfect-is-refiner} holds,
  for the set \(U \cap V\) and the pair of functions \((f \Vert g, \sigma \Vert \tau)\).
  Let \(S,T \in \stacks\) with \(|S| = |T|\).
  By an argument similar to that used in the proof of Lemma~\ref{lem-refiner-intersection}, if \(|f(S)| \neq |\sigma(T)|\), then the condition holds, so suppose otherwise. Then
  \begin{align*}
    (U \cap V) \cap \Iso{S}{T}
      & = (U \cap \Iso{S}{T}) \cap (V \cap \Iso{S}{T}) \\
      & = \Iso{S \Vert f(S)}{T \Vert \sigma(T)} \cap \Iso{S \Vert g(S)}{T \Vert \tau(T)} \\
      & = \Iso{S}{T} \cap \Iso{f(S)}{\sigma(T)} \cap \Iso{g(S)}{\tau(T)} \\
      & = \Iso{S \cap (f \Vert g)(S)}{T \cap (\sigma \Vert \tau)(T)}.
      \qedhere
  \end{align*}
\end{proof}

\section{Perfect refiners for stabilisers and transporter sets}\label{sec-perfect-stab-transp}

Many computations that are commonly performed with backtrack search can be formulated as stabiliser or transporter problems.
This includes computing normalisers, determining subgroup conjugacy, solving graph isomorphism, and finding sets that are phrased in such terms, like set stabilisers.

Therefore, in order to most successfully solve such problems with backtrack search organised around stacks in \(\mathcal{O}\), we require a way of constructing refiners for the appropriate stabilisers and transporter sets.
These refiners should be cheap to compute, and ideally, they should be perfect if possible.

Lemma~\ref{lem-refiner-constant} constructively shows that this task can be reduced to translating the given stabiliser or transporter problem into one concerning stacks in \(\mathcal{O}\).
For a problem that is already given in terms of \(\stacks\), or at least in terms of \(\mathcal{O}\), it follows that little to no translation is needed; this is stated explicitly in the following immediate corollary to Lemma~\ref{lem-refiner-constant}.

\begin{cor}\label{cor-perfect-O-stacksO}
  For each \(x \in \mathcal{O}\) and \(S \in \stacks\), let \([x]\) be the stack with unique entry \(x\), and let \(f_{S}\) be the constant function in \(\stacks^{\stacks}\) with image \(S\).
  Let \(x,y \in \mathcal{O}\) and \(S,T \in \stacks\).
  Then \((f_{[x],}f_{[y]})\) is a perfect refiner for \(\Iso{x}{y}\), and
  \((f_S,f_T)\) is a perfect refiner for \(\Iso{S}{T}\).
\end{cor}

For other transporter problems, more work is required in order to apply Lemma~\ref{lem-refiner-constant}, i.e.\ to construct stacks in \(\mathcal{O}\) whose own transporter set closely contains (and ideally equals) the given one. The analogous statement holds for stabiliser problems.
While this translation can be done on an ad-hoc basis, it is desirable to instead have a systematic method that works for a whole class of objects that we wish to stabilise and to compute transporter sets of.
To that end, we present Theorem~\ref{thm-perfect-stabilisers-and-transporters}.

\begin{defn}\label{defn-invariant}
  Let \(X\) and \(Y\) be sets on which \(\SOm\) acts, and let \(f : X \to Y\) be a function. Then \(f\) is called \emph{\(\SOm\)-invariant} if \(f(x^{g}) = {f(x)}^{g}\) for all \(x \in X\) and \(g \in G\).
\end{defn}

\begin{thm}\label{thm-perfect-stabilisers-and-transporters}
  Let \(\Sigma\) and \(\Theta\) be sets on which \(\SOm\) acts,
  let \(\pi : \Sigma \to \Theta\) be a
  \(\SOm\)-invariant function,
  and let \(x, y \in \Sigma\).
  Then the following statements hold:
  \begin{enumerate}[label=\emph{(\roman*)},ref=\textrm{(\roman*)}]
    \item\label{item-transp-refiner}
    \(\Iso{x}{y} \subseteq \Iso{\pi(x)}{\pi(y)}\).
    In particular, \(\Auto{x} \leq \Auto{\pi(x)}\),

    and any refiner for \(\Iso{\pi(x)}{\pi(y)}\) is a refiner for \(\Iso{x}{y}\).

    \item\label{item-transp-perfect}
    If \(\pi\) is injective, then \(\Iso{x}{y} = \Iso{\pi(x)}{\pi(y)}\);
    in particular, \(\Auto{x} = \Auto{\pi(x)}\),

    and any perfect refiner for \(\Iso{\pi(x)}{\pi(y)}\) is a perfect refiner for \(\Iso{x}{y}\).
  \end{enumerate}
\end{thm}

\begin{proof}
  If \(g \in \Iso{x}{y}\), then \({\pi(x)}^{g} = \pi(x^{g}) = \pi(y)\),
  i.e.\ \(g \in \Iso{\pi(x)}{\pi(y)}\).
  Therefore~\ref{item-transp-refiner} holds.
  Suppose that \(\pi\) is injective.
  If \(g \in \Iso{\pi(x)}{\pi(y)}\), then \({\pi(x)}^{g} = \pi(x^{g})\) and \({\pi(x)}^{g} = \pi(y)\), and so the injectivity of \(\pi\) implies that \(x^{g} = y\).
  Therefore~\ref{item-transp-perfect} holds.
\end{proof}

Theorem~\ref{thm-perfect-stabilisers-and-transporters} suggests a strategy for systematically constructing refiners for the stabilisers and transporter sets of objects in a set on which \(\SOm\) acts, which is the foundation of our approach in Section~\ref{sec-examples-perfect}.
More precisely, given such a set \(\Sigma\), we identify another set \(\Theta\) on which \(\SOm\) acts, and for which we have identified refiners for all stabilisers and transporter sets.
We then define an injective \(\SOm\)-invariant function \(\pi : \Sigma \to \Theta\).
The desired refiners are then inherited via \(\pi\) as described in Theorem~\ref{thm-perfect-stabilisers-and-transporters}, with Corollary~\ref{cor-perfect-O-stacksO} providing the base of this recursive procedure.

\begin{remark}\label{rmk-the-one-chris-wanted}
  Let \(\mathcal{O}\) and \(\mathcal{O}'\) be sets on which \(\SOm\) acts, let \(\pi : \stacks \to \Stacks{\mathcal{O}'}\) be an injective \(\SOm\)-invariant function, and let \(U \subseteq \SOm\).
  Then Lemmas~\ref{lem-refiner-constant} and~\ref{lem-when-perfect-refiners-exist} and Theorem~\ref{thm-perfect-stabilisers-and-transporters}\ref{item-transp-perfect}
  together constructively show that if \(U\) has a perfect refiner in backtrack search organised around \(\stacks\), then \(U\) also has a perfect refiner in backtrack search organised around \(\Stacks{\mathcal{O}'}\).
\end{remark}

In Section~\ref{sec-examples-perfect}, we also make frequent use of the following lemma, which shows that the \(\Vert\) operation can be used to construct refiners for the stabilisers and transporter sets of lists.
This means that lists do need not to be considered separately from the objects that they contain.

\begin{lemma}\label{lem-transporter-list-of}
  Let \(n \in \N\) with \(n \geq 2\).
  For each \(i \in \{1,\ldots,n\}\), let \(x_i\) and \(y_i\) be objects from a set on which \(\SOm\) acts, and let \((f_{L,i}, f_{R,i})\) be a refiner for \(\Iso{x_i}{y_i}\).
  Then \((f_{L,1} \Vert \cdots \Vert f_{L,n},\ f_{R,1} \Vert \cdots \Vert f_{R,n})\) is a refiner for \(\Iso{[x_1,\ldots,x_n]}{[y_1,\ldots,y_n]}\).
  Moreover, if each refiner \((f_{L,i}, f_{R,i})\) is perfect, then the resulting refiner is perfect, too.
\end{lemma}

\begin{proof}
  Since \(\Iso{[x_1,\ldots,x_n]}{[y_1,\ldots,y_n]} = \bigcap_{i = 1}^{n}
  \Iso{x_i}{y_i}\), the result follows from
  Lemmas~\ref{lem-refiner-intersection}
  and~\ref{lem-perfect-refiner-intersection}.
\end{proof}

\begin{cor}\label{cor-perfect-list-of}
  Let \(\Sigma\) be a set on which \(\SOm\) acts.
  Then there exist perfect refiners for all stabilisers and for all transporter sets of objects in \(\Sigma\) if and only if there exist perfect refiners for all stabilisers and transporter sets of finite lists in \(\Sigma\).
\end{cor}

\section{Extended graph backtracking, with extra vertices}\label{sec-extended}

With graph backtracking,
a search in \(\SOm\) is organised around stacks of labelled digraphs
on the vertex set \(\Omega\).
Although this technique gives significantly smaller search sizes in some
cases~\cite[Section~9]{GB_published},
and enables some important perfect refiners
(Section~\ref{sec-perfect-graph-BT}),
the requirement that the digraphs have vertex set \(\Omega\) limits the
possibilities for perfect refiners.
This is shown explicitly in Corollary~\ref{cor-graph-stab-2closed}.

In this section, we introduce an extension to graph backtracking, which accommodates digraphs that are allowed to have additional vertices.
The technique requires only a small adjustment to the theory of graph
backtracking, and has already been implemented in \textsc{Vole}~\cite{vole}.
In fact the necessary concepts from \cite{GB_published} extend naturally, and re-stating and re-proving the corresponding definitions and results do not give any new insights, which is why we do not include these technical details in this article.
Instead, we discuss the extended graphs themselves, and their stacks, in much detail, because this is necessary for formulating and examining refiners in this setting.

We require some additional definitions and notation for the rest of this article.

We fix \(\Lambda\) as an infinite well-ordered set containing \(\Omega\),
where \(\alpha < \beta\) for all \(\alpha \in \Omega\) and \(\beta \in \Lambda \setminus \Omega\).
In practice, and in the forthcoming examples, we use \(\Omega \coloneqq \{1,\ldots,n\}\) for some \(n \in \N\), and \(\Lambda \coloneqq \N\).

For any subset \(V\) of \(\Lambda\) that contains \(\Omega\),
we regard \(\SOm\) and \(\Sym{V \setminus \Omega}\) as subgroups of \(\Sym{V}\),
by identifying each permutation of \(V\) that fixes \(V \setminus \Omega\) pointwise with its restriction to \(\Omega\),
and identifying each permutation of \(V\) that fixes \(\Omega\) pointwise with its restriction to \(V \setminus \Omega\).
In this way, \(\SOm\) and \(\Sym{V \setminus \Omega}\) inherit from \(\SOm\) an action on the set of all labelled digraphs on \(V\).
Furthermore, since we regard \(\SOm\) and \(\Sym{V \setminus \Omega}\) as subgroups of 
\(\Sym{V}\), they 
commute, and this means that 
\(\SOm\) permutes the set of orbits of \(\Sym{V \setminus \Omega}\) on labelled digraphs on \(V\).
For a labelled digraph \(\Gamma\) on \(V\), we use the notation \(\ov{\Gamma}\) for the orbit of \(\Gamma\) under \(\Sym{V \setminus \Omega}\).
The action of \(\SOm\) is then given by \({(\ov{\Gamma})}^{g} \coloneqq \ov{\Gamma^{g}}\) for all labelled digraphs \(\Gamma\) on \(V\) and all \(g \in \SOm\).

We define an \emph{extended graph} on \(\Omega\)
to be an orbit of \(\Sym{V \setminus \Omega}\) on the set of labelled digraphs on \(V\),
for some finite set \(V\) with \(\Omega \subseteq V \subseteq \Lambda\).
This will be illustrated in Example~\ref{ex-extended-graph}.
An \emph{extended graph stack} on \(\Omega\) is any list of extended graphs on \(\Omega\).
We remark that different entries in  an extended graph stack are allowed to be defined in relation to different subsets of \(\Lambda\).
\emph{Extended graph backtracking} in \(\SOm\) is then backtrack search organised around extended graph stacks on \(\Omega\).

The results of Sections~\ref{sec-refiners}--\ref{sec-perfect-stab-transp} hold in the setting of extended graph backtracking in \(\SOm\).

\begin{example}\label{ex-extended-graph}
  Let \(\Omega \coloneqq \{1,\ldots,4\}\) and \(V \coloneqq \{1,\ldots,6\}\).
  Depicted in Figure~\ref{fig-extended-graph} are labelled digraphs \(\Gamma\) and \(\Gamma^{(5\,6)}\) on \(V\),
  which comprise an orbit of \(\Sym{V \setminus \Omega}\) on the set of labelled digraphs on \(V\).
  Therefore \(\ov{\Gamma} = \{\Gamma, \Gamma^{(5\,6)}\}\) is an extended graph on \(\Omega\).
  Note that \(\Auto{\ov{\Gamma}} = \< (1\,4), (1\,2)(3\,4) \> \),
  which is the setwise stabiliser of \(\{\{1,4\},\{2,3\}\}\) in \(\SOm\).
  In Section~\ref{sec-perfect-set-of-sets} we will discuss a similar construction.
\end{example}

\begin{figure}[!ht]
  \small
  \centering
  \begin{tikzpicture}
    \tikzstyle{white}=[circle, draw=black,]
    \tikzstyle{black}=[circle, draw=white, fill=black!100, text=white]
    \node[white] (1) at (0, 1.8)    {1};
    \node[white] (2) at (1, 1.8)    {2};
    \node[white] (3) at (1, 0)      {3};
    \node[white] (4) at (0, 0)      {4};
    \node[black] (5) at (-0.5, 0.9) {5};
    \node[black] (6) at (1.5, 0.9)  {6};
    \node at (0.5, -0.5) {\(\Gamma\)};

    \arcSym{1}{4}
    \arcSym{2}{3}
    \arc{1}{5}
    \arc{4}{5}
    \arc{2}{6}
    \arc{3}{6}
  \end{tikzpicture}
  \qquad\qquad
  \begin{tikzpicture}
    \tikzstyle{white}=[circle, draw=black,]
    \tikzstyle{black}=[circle, draw=white, fill=black!100, text=white]

    \node[white] (1) at (0, 1.8) {1};
    \node[white] (2) at (1, 1.8) {2};
    \node[white] (3) at (1, 0) {3};
    \node[white] (4) at (0, 0) {4};
    \node[black] (5) at (-0.5, 0.9) {6};
    \node[black] (6) at (1.5, 0.9)  {5};
    \node at (0.5, -0.5) {\(\Gamma^{(5\,6)}\)};

    \arcSym{1}{4}
    \arcSym{2}{3}
    \arc{1}{5}
    \arc{4}{5}
    \arc{2}{6}
    \arc{3}{6}
  \end{tikzpicture}
  \caption{\label{fig-extended-graph}
  The digraphs \(\Gamma\) and \(\Gamma^{(5\,6)}\) from Example~\ref{ex-extended-graph}.
  }
\end{figure}
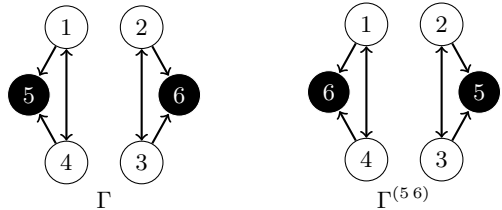

\begin{lemma}\label{lem-extended-stabilises-restriction}
  Let \(A\) and \(B\) be extended graphs on \(\Omega\), defined relative to subsets \(U\) and \(V\) of \(\Lambda\), respectively, and let \(\Gamma \in A\) and \(\Delta \in B\) be arbitrary.
  If \(U \neq V\), then \(\Iso{U}{V} = \varnothing\). Otherwise,
  \(\Iso{A}{B}\) is the restriction of
  \(T \coloneqq \set{x \in \Sym{V}}{x\ \text{preserves}\ \Omega\ \text{setwise, and}\ \Gamma^{x} = \Delta}\) to \(\Omega\).
\end{lemma}

\begin{proof}
  Suppose that \(U = V\) and let \(g\) be the restriction to \(\Omega\) of some \(x \in T\).
  Then since \({g}^{-1} x \in \Sym{V \setminus \Omega}\), it follows that
  \(A^{g} = \ov{\Gamma}^{g} = \ov{\Gamma^{g}} = \ov{\Gamma^{g (g^{-1} x)}} = \ov{\Gamma^{x}} = \ov{\Delta} = B\).
  Conversely, if \(g \in \Iso{A}{B} = \Iso{\ov{\Gamma}}{\ov{\Delta}}\), then
  \( \Gamma^{g} = \Delta^{h}\) for some \(h \in \Sym{V \setminus \Omega}\).
  Thus \(g\) is the restriction of \(g h^{-1} \in T\) to \(\Omega\).
\end{proof}

Lemma~\ref{lem-extended-stabilises-restriction} implies that we do not need to enumerate whole orbits of labelled digraphs in order to compute stabilisers and transporter sets of extended graphs.
In practice, we construct extended graphs that consist of labelled digraphs where the labels used for vertices in \(\Omega\) are never used for vertices in \(\Lambda \setminus \Omega\).
This guarantees that any permutation that maps one such labelled digraph to another necessarily preserves \(\Omega\) as a set.
Therefore, in the notation of Lemma~\ref{lem-extended-stabilises-restriction}, the set \(T\) is simply the transporter set in \(\Sym{V}\) from \(\Gamma\) to \(\Delta\).

\section{The groups and cosets with perfect refiners}\label{sec-hierarchy}

By Lemmas~\ref{lem-when-perfect-refiners-exist} and~\ref{lem-perfect-refiner-group-coset}, the task of finding a perfect refiner for a subgroup of \(\SOm\), or for a coset of a subgroup,
is equivalent to finding a stack whose stabiliser in \(\SOm\) is the corresponding group.
In this section, and for each setting of backtracking, we classify the groups (and correspondingly the cosets of subgroups) for which there exists a perfect refiner.
This information is summarised in Table~\ref{table-perfect-groups}.

It follows from these classifications that, in a sense, partition backtracking inherits the perfect refiners of classical backtracking, that graph backtracking inherits the perfect refiners of partition backtracking, and that extended graph backtracking inherits the perfect refiners of graph backtracking.
This can be shown constructively: it is straightforward to map lists of points to lists of ordered partitions, and lists of ordered partitions to lists of labelled digraphs, in injective \(\SOm\)-invariant ways.
As mentioned in Remark~\ref{rmk-the-one-chris-wanted}, perfect refiners can then be translated via these maps.

\begin{table}[!ht]
  \centering
  \begin{tabular}{lll}
    Backtracking & The subgroups of \(\SOm\), and their cosets, with perfect refiners & Reference \\
    \hline
    Classical
      & \(\Sym{\Sigma}\) for any subset \(\Sigma \subseteq \Omega\).
      & ---
      \\
    Partition
      & \(\Sym{\Sigma_{1}} \times \cdots \times \Sym{\Sigma_{k}}\)
        for any partition \(\{\Sigma_{1}, \ldots, \Sigma_{k}\}\) of \(\Omega\).
      & ---
      \\
    Graph
      & Any 2-closed subgroup of \(\SOm\).
      & Cor.~\ref{cor-graph-stab-2closed}
      \\
    Extended graph
      &
      Any subgroup of \(\SOm\).
      & Cor.~\ref{lem-extended-perfect}
      \\
  \end{tabular}
  \caption{
    This table shows the subgroups of \(\SOm\), and the cosets of subgroups of \(\SOm\), for which there exists a perfect refiner, in various kinds of backtracking.
    For a subset \(\Sigma \subseteq \Omega\), we identify \(\Sym{\Sigma}\) with the pointwise stabiliser of \(\Omega \setminus \Sigma\) in \(\SOm\), and for a partition \(\{\Sigma_{1}, \ldots, \Sigma_{k}\}\) of \(\Omega\), we identify \(\Sym{\Sigma_{1}} \times \cdots \times \Sym{\Sigma_{k}}\) with the corresponding internal direct product in \(\SOm\).
  }\label{table-perfect-groups}
\end{table}

The classifications in Table~\ref{table-perfect-groups} for classical backtracking (stabilisers in \(\SOm\) of lists in \(\Omega\)) and partition backtracking (stabilisers in \(\SOm\) of lists of ordered partitions of \(\Omega\)) are trivial to verify.

For graph backtracking, we require the notion of the 2-closure of a permutation group.

\begin{defn}[\mbox{2-closure, 2-closed; cf.~\cite[Section~3.2]{DixonMortimer}}]
  The \emph{2-closure} of a group \(G \leq \SOm\) is the largest subgroup of \(\SOm\) with the same orbits on \(\Omega \times \Omega\) as \(G\); it is the pointwise stabiliser in \(G\) of the set of all orbital graphs of \(G\).
  A \emph{2-closed} group is one that is equal to its 2-closure.
\end{defn}

The 2-closure of any subgroup of \(\SOm\) that acts at least 2-transitively on \(\Omega\) is \(\SOm\).  Therefore no proper subgroup of \(\SOm\) acts at least 2-transitively on \(\Omega\) and is 2-closed.

\begin{lemma}\label{lem-graph-stab-2closed}
  Let \(G\) be a subgroup of \(\SOm\).
  Then \(G\) is the stabiliser of a labelled digraph stack on \(\Omega\) if and only if \(G\) is 2-closed.
\end{lemma}

\begin{proof}
  \((\Leftarrow)\) Labelling all vertices and arcs of a digraph with a fixed label preserves its stabiliser. Thus, if \(S\) is a stack of labelled digraphs formed from all the orbital graphs of \(G\) in this way, then \(G = \Auto{S}\).

  \((\Rightarrow)\)
  By~\cite[Lemma~3.6]{GB_published}, we may assume that \(G = \Auto{\Gamma}\) for a labelled digraph \(\Gamma\) on \(\Omega\).
  Let \(A_1, \ldots, A_k \subseteq \Omega \times \Omega\) and \(B_1, \ldots, B_l \subseteq \Omega\) be the orbits of \(G\) on the sets of arcs and vertices of \(\Gamma\), respectively, for some \(k \in \N \cup \{0\}\) and \(l \in \N\).
  For each \(i \in \{1,\ldots,k\}\) and \(j \in \{1,\ldots,l\}\), we define digraphs \(\Delta_i \coloneqq (\Omega,A_i)\) and \(\Sigma_j \coloneqq (\Omega,\set{(b, b)}{b \in B_j})\).
  Each such digraph is an orbital graph of \(G\), and so its stabiliser contains \(G\). Furthermore, if a permutation \(g \in \SOm\) stabilises each of these digraphs
  pointwise, then by construction, \(g\) maps each vertex of \(\Gamma\) to a vertex with the same label, and it maps each arc of \(\Gamma\) to an arc with the same label, and so \(g \in \Auto{\Gamma}\).
  Thus we have proved that
  \[
    G
        \leq
    \left( \bigcap_{i = 1}^{k} \Auto{\Delta_i} \right)
      \cap
    \left( \bigcap_{j = 1}^{l} \Auto{\Sigma_j} \right)
        \leq
    \Auto{\Gamma}
        =
    G.
  \]
  Since \(G\) is the pointwise stabiliser of a subset of its orbital graphs, it is 2-closed.
\end{proof}

\begin{cor}\label{cor-graph-stab-2closed}
  In graph backtracking, there exists a perfect refiner for a group \(G \leq \SOm\) if and only if \(G\) is 2-closed.
\end{cor}

For extended graph backtracking, we require the following result.

\begin{prop}[\!\!\cite{Bouwer69,MathOverflow}]
  \label{prop-arbitrary-group-graph}
  Let \(G \leq \SOm\).
  There exists a finite set \(V\) containing \(\Omega\)
  and a labelled digraph \(\Gamma\) on \(V\),
  such that the stabiliser of \(\Gamma\) in \(\Sym{V}\) preserves \(\Omega\) setwise,
  and the restriction of this stabiliser to \(\Omega\) is \(G\).
\end{prop}

\begin{lemma}\label{lem-extended-perfect}
  Every subgroup of \(\SOm\) has a perfect refiner in extended graph backtracking.
\end{lemma}

\begin{proof}
  Let \(G \leq \SOm\), and let \(\Gamma\) be a finite labelled digraph of the kind described in Proposition~\ref{prop-arbitrary-group-graph} for \(G\).
  By renaming the vertices that are not in \(\Omega\), we may assume that the vertex set of \(\Gamma\) is a subset of the infinite set \(\Lambda\) that we use throughout this article.
  By Lemma~\ref{lem-extended-stabilises-restriction} and Proposition~\ref{prop-arbitrary-group-graph}, the extended graph stack \([\ov{\Gamma}]\) has stabiliser \(G\) in \(\SOm\).
  The result follows by Lemma~\ref{lem-when-perfect-refiners-exist}.
\end{proof}

We primarily include Lemma~\ref{lem-extended-perfect} for its theoretical interest.
Given a group \(G\), the number of additional vertices required to construct a witness for the result in Proposition~\ref{prop-arbitrary-group-graph} may be of similar size to \(|G|\), which would normally be impractically large for use in a backtrack search. We are therefore also interested in how many extra vertices are required to represent different groups.

\section{Examples of perfect refiners}\label{sec-examples-perfect}

In this section,
we describe some refiners for the stabilisers and transporter sets of some commonly-occurring objects on which \(\SOm\) acts.
In most cases, we prove that the refiners are perfect, making frequent use of  Corollary~\ref{cor-perfect-O-stacksO}, Theorem~\ref{thm-perfect-stabilisers-and-transporters}, and Lemma~\ref{lem-transporter-list-of}.
The results for perfect refiners are summarised in Table~\ref{table-perfect-stab}.
As discussed in Section~\ref{sec-hierarchy}, a perfect refiner for a subset of \(\SOm\) in classical backtracking can be easily converted into a perfect refiner for the same set in partition backtracking, and so on.
In this hierarchical sense, each class of objects in Table~\ref{table-perfect-stab} appears by the `least' kind of backtracking in which all of their stabilisers and transporter sets have perfect refiners.

\begin{table}[!ht]
  \centering
  \begin{tabular}{llr}
    Backtracking & Stabilisers/transporter sets in \(\SOm\) with perfect refiners & Reference \\
    \hline
    Classical
      & Point in \(\Omega\).
        & Corollary~\ref{cor-perfect-O-stacksO} \\
      & List in \(\Omega\).
        & Corollary~\ref{cor-perfect-O-stacksO} \\
    \hline
    Partition
      & Ordered partition of \(\Omega\).
        & Corollary~\ref{cor-perfect-O-stacksO} \\
      & List of ordered partitions of \(\Omega\).
        & Corollary~\ref{cor-perfect-O-stacksO} \\
      & Subset of \(\Omega\).
        & Section~\ref{sec-subset} \\
      & List of subsets of \(\Omega\).
        & Lemma~\ref{lem-transporter-list-of} \\
      & Set of subsets of \(\Omega\) with pairwise distinct sizes.
        & Section~\ref{sec-set-of-subset-distinct-sizes} \\
    \hline
    Graph
      & Labelled digraph on \(\Omega\).
        & Corollary~\ref{cor-perfect-O-stacksO} \\
      & Graph or digraph on \(\Omega\).
        & Section~\ref{sec-perfect-graph-auto-iso} \\
      & Set of nonempty pairwise disjoint subsets of \(\Omega\).
        & Section~\ref{sec-perfect-set-disjoint-sets} \\
      & Unordered partition of \(\Omega\).
        & Section~\ref{sec-perfect-set-disjoint-sets} \\
      & Permutation under conjugation (elt.\ centraliser/conjugacy).
        & Section~\ref{sec-perfect-perm-conj} \\
      & List of permutations under conj.\ (subset/subgroup centraliser).
        & Lemma~\ref{lem-transporter-list-of} \\
    \hline
    Extended graph
      & Set of subsets of \(\Omega\).
        & Section~\ref{sec-perfect-set-of-sets} \\
      & Set of lists in \(\Omega\).
        & Section~\ref{sec-perfect-set-of-lists} \\
      & Set of (labelled) graphs or digraphs on \(\Omega\).
        & Section~\ref{sec-perfect-set-of-graphs} \\
      & Set of lists of (labelled) graphs or digraphs on \(\Omega\).
        & Section~\ref{sec-perfect-set-of-graph-stacks} \\
      & 2-closed subgroup under conjugation (normaliser/conjugacy).
        & Section~\ref{sec-perfect-2-closed-normaliser} \\
  \end{tabular}
  \caption{\label{table-perfect-stab}
    For the various kinds of backtracking in \(\SOm\), this table gives some objects on which \(\SOm\) acts, such that all stabilisers and transporter sets in \(\SOm\) of these objects have perfect refiners. A constructive argument is given in the reference. This table is not meant to be exhaustive.
  }
\end{table}

A given kind of backtracking may not have perfect refiners for \emph{all} stabilisers and transporter sets of some class of objects, but it may have perfect refiners for those of an important subclass of the objects.
In order to optimise implementations, it is important to understand these special cases well, but a thorough investigation of this topic is beyond the scope of this article.
Nevertheless, to demonstrate the principle, we consider sets of subsets of \(\Omega\) with pairwise distinct sizes
in partition backtracking
(Section~\ref{sec-set-of-subset-distinct-sizes}), and sets of nonempty pairwise disjoint subsets of \(\Omega\)
in graph backtracking
(Section~\ref{sec-perfect-set-disjoint-sets}), even though only extended graph backtracking has perfect refiners for all stabilisers and transporter sets of sets of subsets of \(\Omega\) (Section~\ref{sec-perfect-set-of-sets}).

\subsection{Examples of perfect refiners in partition backtracking}\label{sec-perfect-partition-BT}

To justify the part of Table~\ref{table-perfect-stab} that relates to partition backtracking, we first show, for each kind of object, and with \(|\Omega| \geq 5\), that not all stabilisers and transporter sets have perfect refiners in classical backtracking.
By Corollary~\ref{cor-perfect-list-of} and Table~\ref{table-perfect-groups}, the following examples suffice:
If \(n \in \N\) with \(n \geq 5\) and \(\Omega \coloneqq \{1,\ldots,n\}\), then the ordered partition \([\{1,2\},\{3,\ldots,n\}]\) of \(\Omega\), the subset \(\{1,2\}\) of \(\Omega\), and the set of subsets \(\left\{\{1,2\},\{3,\ldots,n\}\right\}\) of \(\Omega\) have stabiliser \(\< (1\,2), (3\,4), (3\,\ldots\,n) \> \cong \Cn{2} \times \Sn{n-2}\) in \(\SOm\).

\subsubsection{Stabilisers and transporters of sets of points}\label{sec-subset}

We define a function \(\pi\) from the set of all subsets of \(\Omega\) to the set of all stacks of ordered partitions of \(\Omega\).
We define \(\pi(\varnothing)\) to be the empty stack on \(\Omega\), and \(\pi(\Omega)\) to be the stack consisting of the ordered partition \([\Omega]\), and for any nonempty proper subset \(A\) of \(\Omega\), we define \(\pi(A)\) to be the stack consisting of the ordered partition \([A, \Omega \setminus A]\).
Since \(\pi\) is an injective \(\SOm\)-invariant function, Theorem~\ref{thm-perfect-stabilisers-and-transporters} and Corollary~\ref{cor-perfect-O-stacksO} give the desired perfect refiners.

\subsubsection{Stabilisers and transporters of sets of subsets with pairwise distinct sizes}\label{sec-set-of-subset-distinct-sizes}

We define an injective \(\SOm\)-invariant function \(\sigma\) that maps each set of subsets of \(\Omega\) with pairwise distinct sizes to a list of subsets of \(\Omega\) as follows.
Let \(A\) be any such set.
There exists \(m \in \N \cup \{0\}\) and subsets \(A_{1}, \ldots, A_{m}\) of \(\Omega\), uniquely indexed by increasing size, such that \(A = \{A_{1}, \ldots, A_{m}\}\).
We define \(\sigma(A) \coloneqq [A_{1}, \ldots, A_{m}]\), and remark that a permutation of \(\Omega\) stabilises the set of subsets \(A\) if and only if it stabilises each of the subsets that it contains.
Perfect refiners for the stabilisers and transporters can thus be obtained from \(\sigma\) via Section~\ref{sec-subset}, Lemma~\ref{lem-transporter-list-of}, and Theorem~\ref{thm-perfect-stabilisers-and-transporters}.

\subsection{Examples of perfect refiners in graph backtracking}\label{sec-perfect-graph-BT}

For the kinds of objects listed by graph backtracking in Table~\ref{table-perfect-stab}, it is easy to find examples for all \(|\Omega| \geq 4\) where the stabiliser in \(\SOm\) is not a direct product of symmetric groups, which means that the stabiliser has no perfect refiner in partition backtracking (see Table~\ref{table-perfect-groups}).
It remains to show that there are perfect refiners in graph backtracking for the stabilisers and transporter sets of these objects.

\subsubsection{Graph and digraph automorphisms and isomorphisms}\label{sec-perfect-graph-auto-iso}

We recall our standard notation for graphs from Section \ref{sec-prelim}.
Let \(\pi\) be the function that maps each graph \((\Omega,E)\) to the digraph \((\Omega, \set{(\alpha,\beta), (\beta,\alpha)}{ \{\alpha,\beta\} \in E} )\). This means that each edge \(\{\alpha,\beta\}\) is replaced by the arcs \((\alpha,\beta)\) and \((\beta,\alpha)\).
In addition, we define \(x\) to be some fixed label, and we let \(\sigma\) be the function that maps each digraph on \(\Omega\) to the labelled digraph on \(\Omega\) formed by assigning the label \(x\) to all of its vertices and arcs.
Then \(\pi\) and \(\sigma\) are injective \(\SOm\)-invariant functions, and so Theorem~\ref{thm-perfect-stabilisers-and-transporters} and Corollary~\ref{cor-perfect-O-stacksO} give perfect refiners for all stabilisers and transporter sets in \(\SOm\) of graphs and digraphs on \(\Omega\) in graph backtracking.

\subsubsection{Stabilisers and transporters of sets of nonempty pairwise disjoint sets}\label{sec-perfect-set-disjoint-sets}

We define a function \(\phi\) from the set of all subsets of \(\Omega\) to the set of all digraphs on \(\Omega\) as follows.
Let \(A \coloneqq \{A_1, \ldots, A_k\}\), for some \(k \in \N \cup \{0\}\), be a set of subsets of \(\Omega\),
and define \(\phi(A)\) to be the digraph on \(\Omega\) that consists of a clique (plus loops) on the vertices of each member of \(A\), i.e.
\[
  \phi(A) \coloneqq
  \big(
    \Omega,\
    \set{(\alpha,\beta) \in \Omega\times\Omega}
        {\{\alpha, \beta\} \subseteq A_{i} \ \text{for some}\ i \in \{1,\ldots,k\}}
  \big).
\]
See Figure~\ref{fig-set-disjoint-set} for an example.
Then \(\phi\) is \(\SOm\)-invariant, and so by Theorem~\ref{thm-perfect-stabilisers-and-transporters} and Section~\ref{sec-perfect-graph-auto-iso}, we obtain refiners in graph backtracking for the stabilisers and transporter sets of all sets of subsets of \(\Omega\).
Furthermore, when restricted to sets of nonempty pairwise disjoint subsets of \(\Omega\),
the function \(\phi\) is injective and \(\SOm\)-invariant, and so the corresponding refiners are perfect.

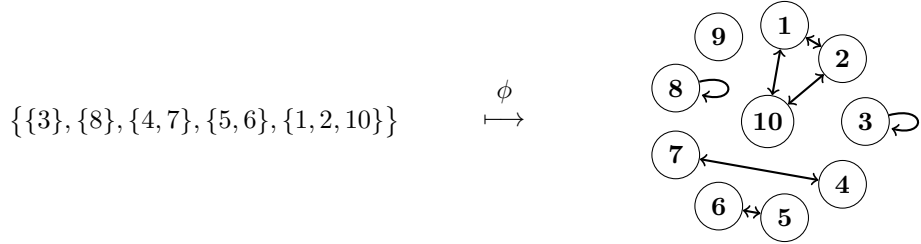
\begin{figure}[!ht]
  \begin{center}
    \begin{tikzpicture}
      \node (11) at (-7.5,0) {\(\big\{ \{3\}, \{8\}, \{4,7\}, \{5,6\}, \{1,2,10\} \big\}\)};

      \node (21) at (-3.5,0.4) {\(\phi\)};
      \node (21) at (-3.5,0) {\(\longmapsto\)};

      \node[circle, draw=black] (10) at (0,0) {\!\(\mathbf{10}\)\!};
      \foreach \x in {1,2,3,4,5,6,7,8,9} {
        \node[circle, draw=black] (\x) at (-\x*40+120:1.3cm) {\(\mathbf{\x}\)};};

        \arcSym{1}{2}
        \arcSym{1}{10}
        \arcSym{2}{10}
        \arcSym{4}{7}
        \arcSym{5}{6}
        \looparcR{3}
        \looparcR{8}

    \end{tikzpicture}
  \end{center}
  \caption[Perfect refiner example for the stabiliser of a set of disjoint subsets]{
    An example of the function \(\phi\) from Section~\ref{sec-perfect-set-disjoint-sets} for the given set of disjoint subsets of \(\Omega \coloneqq \{1,\ldots,10\}\).
    To reduce visual clutter, loops are omitted at vertices of non-singleton subsets.
  }\label{fig-set-disjoint-set}
\end{figure}

\subsubsection{Permutation centraliser and conjugacy, and subgroup centraliser}\label{sec-perfect-perm-conj}

We consider \(\SOm\) acting on itself by conjugation.
Let \(\psi\) be the function from \(\SOm\) to the set of all digraphs on \(\Omega\) where for each \(g \in \SOm\), \(\psi(g) \coloneqq \left(\Omega, \set{(\alpha, \beta) \in \Omega \times \Omega}{\alpha ^ g = \beta}\right) \). See Figure~\ref{fig-perm-centraliser}, where we explain an example for illustration.

\begin{figure}[!ht]
  \begin{center}
    \begin{tikzpicture}
      \node (11) at (-5,0) {\((1\,2)(3\,6\,5)\)};

      \node (21) at (-3,0.4) {\(\psi\)};
      \node (21) at (-3,0) {\(\longmapsto\)};

      \foreach \x in {1,2,3,4,5,6} {
        \node[circle, draw=black] (\x) at (-\x*60+120:1cm) {\(\mathbf{\x}\)};};

      \arcSym{1}{2}
      \arc{3}{6}
      \arc{6}{5}
      \arc{5}{3}
      \looparcL{4}
    \end{tikzpicture}
  \end{center}
  \caption[Perfect refiner for permutation centraliser]{
    An example of the function \(\psi\) from Section~\ref{sec-perfect-perm-conj} for the permutation \((1\,2)(3\,6\,5)\), with \(\Omega \coloneqq \{1,\ldots,6\}\).
    This digraph gives rise to a perfect refiner for the centraliser of \((1\,2)(3\,6\,5)\) in \(\SOm\).
  }\label{fig-perm-centraliser}
\end{figure}
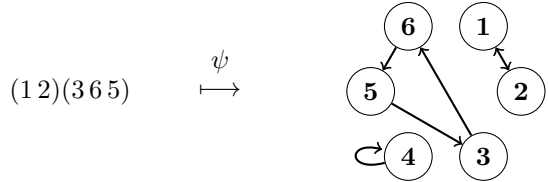

Then \(\psi\) is injective and \(\SOm\)-invariant, giving perfect refiners in graph backtracking for all centralisers in \(\SOm\) of elements of \(\SOm\), and for the transporter sets for permutation conjugacy (see Theorem~\ref{thm-perfect-stabilisers-and-transporters} and Section~\ref{sec-perfect-graph-auto-iso}).
By Lemma~\ref{lem-transporter-list-of}, this gives perfect refiners for all stabilisers of {lists} of permutations in \(\SOm\) under conjugation.
The centraliser in \(\SOm\) of a subgroup of \(\SOm\) is the pointwise stabiliser of any of its generating sets, and thus we can obtain a perfect refiner in graph backtracking for the centraliser in \(\SOm\) of any subgroup of \(\SOm\) given by a generating set.

\subsection{Examples of perfect refiners in extended graph backtracking}\label{sec-perfect-graphsplus}

By Lemma~\ref{lem-extended-perfect}, all stabilisers and transporter sets have perfect refiners in extended graph backtracking, but this knowledge is not necessarily immediately useful:
the extended graphs underpinning this lemma may require impractically many additional vertices, and anyway, the construction requires knowing the stabiliser or transporter set in advance.
In order to avoid too many additional vertices or circular arguments, we wish to construct refiners in extended graph backtracking using only facts about the relevant objects that can be computed cheaply.

In the forthcoming examples, we specify an extended graph
by giving one of the labelled digraphs that it contains.
Therefore, in each case, we must prove that this is well-defined.
The fixed set of vertices \(\Lambda\) is totally ordered, and we always choose vertices from this set in ascending order.
Thus it suffices to show, in each case, that our choices lead to labelled digraphs that differ only by a permutation of their vertices in \(\Lambda \setminus \Omega\).

\subsubsection{Stabilisers and transporters of sets of sets}\label{sec-perfect-set-of-sets}

We define a function \(\phi\) from the set of all sets of subsets of \(\Omega\) to the set of all extended graphs on \(\Omega\).
Let \(A \coloneqq \{A_1, \ldots, A_k\}\) be a set of subsets of \(\Omega\),
where \(k \in \N \cup \{0\}\), indexed arbitrarily,
and let \(V_{k} \coloneqq \{\beta_1, \ldots, \beta_k\} \subseteq \Lambda \setminus \Omega\) comprise the \(k\) least elements of \(\Lambda \setminus \Omega\).
We define \(\Gamma_{A}\) to be the labelled digraph on \(\Omega \cup V_{k}\) with arcs
\(\bigcup_{i = 1}^{k} \set{(\alpha, \beta_{i})}{\alpha \in A_{i}}\),
where vertices in \(\Omega\) are labelled \emph{white}, and all other vertices and arcs are labelled \emph{black}.
Thus there is a \emph{white} vertex in \(\Gamma_{A}\) for each member of \(\Omega\), and a \emph{black} vertex for each member of \(A\); each vertex in \(\Omega\) has arcs to the \emph{black} vertices corresponding to the members of \(A\) that contain it.
We define \(\phi(A) \coloneqq \ov{\Gamma_{A}}\) to be the extended graph containing \(\Gamma_{A}\).

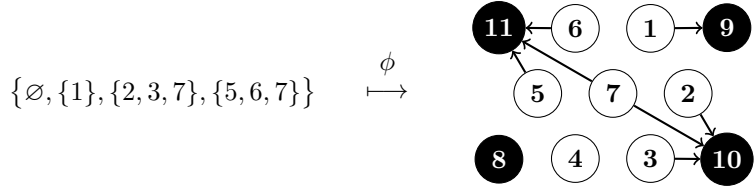
\begin{figure}[!ht]
  \begin{center}
    \begin{tikzpicture}
      \node (11) at (-6,0) {\(\big\{ \varnothing, \{1\}, \{2,3,7\}, \{5,6,7\} \big\}\)};

      \node (21) at (-3,0.4) {\(\phi\)};
      \node (21) at (-3,0) {\(\longmapsto\)};

      \node[circle, draw=black] (7) at (0,0) {\(\mathbf{7}\)};
      \foreach \x in {1,2,3,4,5,6} {
        \node[circle, draw=black] (\x) at (-\x*60+120:1cm) {\(\mathbf{\x}\)};};

      \node[circle, draw=black, text=white, fill=black] (8) at (210:1.75cm) {\(\mathbf{8}\)};
      \node[circle, draw=black, text=white, fill=black] (9) at (30:1.75cm) {\(\mathbf{9}\)};
      \node[circle, draw=black, text=white, fill=black] (10) at (330:1.75cm) {\!\(\mathbf{10}\)\!};
      \node[circle, draw=black, text=white, fill=black] (11) at (150:1.75cm) {\!\(\mathbf{11}\)\!};

      \arc{1}{9}
      \arc{2}{10}
      \arc{3}{10}
      \arc{7}{10}
      \arc{5}{11}
      \arc{6}{11}
      \arc{7}{11}
    \end{tikzpicture}
  \end{center}
  \caption[Perfect refiner example for the stabiliser of a set of sets]{
    A representative labelled digraph of the extended graph defined by the function \(\phi\) from Section~\ref{sec-perfect-set-of-sets}, for the set of subsets \(\big\{ \varnothing, \{1\}, \{2,3,7\}, \{5,6,7\} \big\}\) of \(\Omega \coloneqq \{1,\ldots,7\}\), with \(\Lambda \coloneqq \N\).
  }\label{fig-set-of-sets}
\end{figure}

The only choice involved in constructing a labelled digraph using the method described above is the indexing of the members of \(A\). Therefore the labelled digraphs produced can differ only by a permutation of the names of the vertices outside of \(\Omega\),
and \(\phi\) is well-defined.
It is clear that \(\phi\) is injective and \(\SOm\)-invariant, and so perfect refiners are given by Theorem~\ref{thm-perfect-stabilisers-and-transporters} and Corollary~\ref{cor-perfect-O-stacksO}.

Next, we discuss a situation where there is no perfect refiner in graph backtracking.
Let \(H\) be the subgroup \(\< (1\,2\,3)(4\,5\,6), (1\,2)(3\,5) \>\) of \(\Sym{\{1,\dots,6\}}\).
It can be shown that \(H\) is a proper subgroup of \(\Sym{\{1,\dots,6\}}\) that acts 2-transitively on \(\Omega\), and which is therefore not 2-closed.
If we define \(O\) to be the orbit of \({\{1,2,3\}}\) under \(H\), then \(H\) is exactly the stabiliser of \(O\) in \(\Sym{\{1,\ldots,6\}}\).
Let \(n \in \N\) with \(n \geq 7\) and let \(\Omega \coloneqq \{1,\ldots,n\}\).
It follows that the group \(G \coloneqq H \times \Sym{\{7,\ldots,n\}}\)
(identifying this direct product with a subgroup of \(\SOm\) as in Table~\ref{table-perfect-groups}) is the stabiliser in \(\SOm\) of the set of subsets \(O\), but it is not 2-closed, and therefore it has no perfect refiner in graph backtracking.

\subsubsection{Stabilisers and transporters of sets of lists}\label{sec-perfect-set-of-lists}

We define a function \(\pi\) that maps each set of lists in \(\Omega\) to an extended graph on \(\Omega\).
Let \(A \coloneqq \{A_1, \ldots, A_k\}\) be a set of nonempty lists in \(\Omega\), for some \(k \in \N \cup \{0\}\), indexed arbitrarily. Moreover,
let \(r \coloneqq \prod_{i = 1}^{k} |A_i|\) be the product of the lengths of these lists,
let \(V_{r}\) be a set of the least \(r\) elements of \(\Lambda \setminus \Omega\),
and let \(\rho\) be an arbitrary bijection from \(\set{(i,j)}{i \in \{1,\ldots,k\} \text{\ and\ } j \in \{1,\ldots,|A_{i}|\}}\) to \(V_{r}\).

We define \(\pi(A)\) to be the extended graph that contains the labelled digraph on \(\Omega \cup V_{r}\) with arcs
\[
  \bigset{ \big( \rho(i,j),\ A_{i}[j] \big) }
         { 1 \leq i \leq k,\ 1 \leq j \leq |A_{i}| }
  \cup
  \bigset{ \big( \rho(i,j),\ \rho(i,j+1) \big) }
         { 1 \leq i \leq k,\ 1 \leq j < |A_{i}| },
\]
where the vertices in \(\Omega\) are labelled \emph{white}, and the vertices in \(V_{r}\) and all arcs are labelled \emph{black}.

Each vertex in \(V_{r}\) corresponds via \(\rho\) to a position in a list in \(A\);
the arcs between vertices in \(V_{r}\) encode the ordering of each list,
and the arcs towards vertices in \(\Omega\) encode the entry at each position.

For a set of lists in \(\Omega\) that includes the empty list, \([\ ]\), we proceed as above, except that the labelled digraph has an additional isolated vertex (the next least element of \(\Lambda \setminus \Omega\)) with label \emph{black}.

It is straightforward to see that \(\pi\) is well-defined (once a bijection \(\rho\) is fixed), injective, and \(\SOm\)-invariant,
and so with this method, it is possible to use Theorem~\ref{thm-perfect-stabilisers-and-transporters} and Corollary~\ref{cor-perfect-O-stacksO} to produce perfect refiners in extended graph backtracking for the stabilisers and transporter sets of any sets of lists in \(\Omega\).

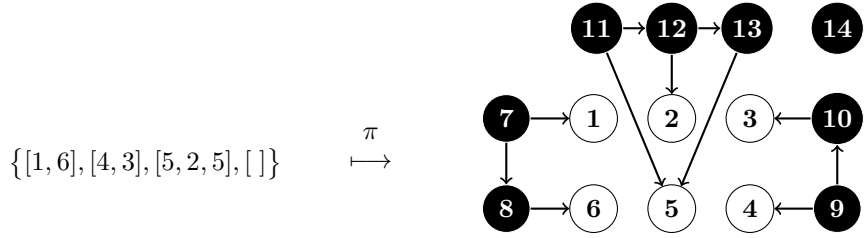
\begin{figure}[!ht]
  \begin{center}
    \begin{tikzpicture}
      \tikzstyle{black}=[circle, draw=white, fill=black, text=white]
      \tikzstyle{white}=[circle, draw=black]

      \node (11) at (-7,0) {\(\big\{ [1,6], [4,3], [5,2,5], [\ ] \big\}\)};

      \node (21) at (-4,0.4) {\(\pi\)};
      \node (22) at (-4,0) {\(\longmapsto\)};

      \foreach \x in {1,3,4,6} {
        \node[white] (\x) at (-\x*60+210:1.2cm) {\(\mathbf{\x}\)};
      };
      \node[white] (2) at (90:0.6cm) {\(\mathbf{2}\)};
      \node[white] (5) at (-90:0.6cm) {\(\mathbf{5}\)};

      \node[black] (7) at (-2.2,0.6) {\(\mathbf{7}\)};
      \node[black] (8) at (-2.2,-0.6) {\(\mathbf{8}\)};

      \arc{7}{1}
      \arc{8}{6}
      \arc{7}{8}

      \node[black] (9) at (2.2,-0.6) {\(\mathbf{9}\)};
      \node[black] (10) at (2.2,0.6) {\!\(\mathbf{10}\)\!};

      \arc{9}{4}
      \arc{10}{3}
      \arc{9}{10}

      \node[black] (11) at (-1,1.8) {\!\(\mathbf{11}\)\!};
      \node[black] (12) at (0,1.8) {\!\(\mathbf{12}\)\!};
      \node[black] (13) at (1,1.8) {\!\(\mathbf{13}\)\!};

      \arc{11}{12}
      \arc{12}{13}
      \arc{11}{5}
      \arc{12}{2}
      \arc{13}{5}

      \node[black] (14) at (2.2,1.8) {\!\(\mathbf{14}\)\!};

    \end{tikzpicture}
  \end{center}
  \caption[Perfect refiner example for the stabiliser of a set of lists]{
    A representative labelled digraph of the extended graph defined by the function \(\pi\) from Section~\ref{sec-perfect-set-of-lists}, for the set of lists \(A \coloneqq \big\{ [1,6], [4,3], [5,2,5], [\ ] \big\}\) in \(\Omega \coloneqq \{1,\ldots,6\}\), with \(\Lambda \coloneqq \N\).
    Its stabiliser in \(\Sym{\{1,\ldots,14}\}\) preserves \(\Omega\) setwise, and the restriction of this stabiliser to \(\Omega\) is \(\Auto{A}\).
  }\label{fig-set-of-lists}
\end{figure}

In combination with Lemma~\ref{lem-graph-stab-2closed}, the following lemma implies that for all sets \(\Omega\) with \(|\Omega| > 3\), there exist sets of lists in \(\Omega\) whose stabilisers and transporter sets do not have perfect refiners in graph backtracking.
This result also gives a different constructive proof of Lemma~\ref{lem-extended-perfect}.
However, the method described in this section is impractical for constructing a perfect refiner for an arbitrary subgroup \(G \leq \SOm\), since it would produce a labelled digraph with \((|G|+1)|\Omega|\) vertices.

\begin{lemma}\label{lem-every-group-is-a-set-of-lists-stab}
  Every subgroup of \(\SOm\) is the stabiliser in \(\SOm\) of a set of lists in \(\Omega\).
\end{lemma}

\begin{proof}
  Let \(G \leq \SOm\) and \(A\) be an enumeration of \(\Omega\).
  Then \(G\) is the stabiliser in \(\SOm\) of \(A^{G}\).
\end{proof}

\subsubsection{Stabilisers and transporters of sets of graphs, digraphs, or labelled digraphs}\label{sec-perfect-set-of-graphs}

We define a function \(\psi\) from the set of all sets of labelled digraphs on \(\Omega\) to the set of all extended graphs on \(\Omega\) that is injective and \(\SOm\)-invariant.
Therefore \(\psi\) can be combined with the functions of Section~\ref{sec-perfect-graph-auto-iso}, via Theorem~\ref{thm-perfect-stabilisers-and-transporters}, to give perfect refiners in extended graph backtracking for the stabilisers and transporter sets of sets of graphs or digraphs on \(\Omega\) that are not necessarily labelled.

First, we fix labels \(\#\) and \(\circledast\) that are not allowed to be used as labels in any labelled digraph on \(\Omega\).

Let \(A \coloneqq \{\Gamma_{1}, \ldots, \Gamma_{k}\}\) be a set of labelled digraphs on \(\Omega\), for some \(k \in \N \cup \{0\}\), indexed arbitrarily.

Roughly speaking, we build a new labelled digraph from a disjoint union of copies of the \(k\) members of \(A\), retaining labels,
and adding arcs that anchor each copy, and arcs that maintain the correspondences between \(\Omega\) and the vertices of the copies.
We give an example in Figure~\ref{fig-set-of-graphs-to-extended-graph} and discuss the precise construction 
in the following paragraph.

Let \(\{V_1, \ldots, V_k\}\) be an arbitrary partition of the least \(k |\Omega|\) elements of \(\Lambda \setminus \Omega\) into \(k\) parts of size \(|\Omega|\), and for each \(i \in \{1,\ldots,k\}\), let \(\tau_{i}\) be an arbitrary bijection from \(\Omega\) to \(V_{i}\).
In addition, let \(X \coloneqq \{x_1, \ldots, x_k\}\) be the least \(k\) elements of \(\Lambda \setminus (\Omega \cup V_1 \cup \cdots \cup V_k)\), indexed arbitrarily.
Then we define \(\psi(A)\) to be the extended graph that contains the labelled digraph on \(\Omega \cup V_1 \cup \cdots \cup V_k \cup X\) with arcs
\[
  \bigcup_{i = 1}^{k}
    \left(
      \bigset{ ( \tau_{i}(\alpha), \alpha),\ ( \tau_{i}(\alpha), x_{i}  )}
             { \alpha \in \Omega }
      \cup
      \bigset{ ( \tau_{i}(\alpha), \tau_{i}(\beta) )}
             { ( \alpha, \beta ) \ \text{is an arc of}\ \Gamma_{i} }
    \right),
\]
where vertices in \(\Omega\) and arcs ending in \(\Omega\) are labelled \(\#\),
vertices in \(X\) and arcs ending in \(X\) are labelled \(\circledast\),
and for all \(\alpha,\beta \in \Omega\) and \(i \in \{1,\ldots,k\}\), the vertex \(\tau_{i}(\alpha)\) has the label of \(\alpha\) in \(\Gamma_{i}\), and the arc \((\tau_{i}(\alpha), \tau_{i}(\beta))\) (if present) has the label of \((\alpha,\beta)\) in \(\Gamma_{i}\).

\begin{figure}[!ht]
  \small
  \centering
  \begin{tikzpicture}
    \tikzstyle{white}=[circle, draw=black,]
    \tikzstyle{black}=[circle, draw=white, fill=black!100, text=white]
    \tikzstyle{grey}=[circle, draw=black, fill=gray!25]

    \node[black] (21) at (-8, 1.733) {1};
    \node[black] (22) at (-6.267, 1.733) {2};
    \node[black] (23) at (-7.134, 0) {3};
    \node (20) at (-7.134,1) {\(\Gamma_{1}\)};
    \node (40) at (-5.5,1) {and};
    \arc{21}{22}
    \arc{22}{23}
    \arc{23}{21}
    \node[black] (31) at (-5, 1.733) {1};
    \node[black] (32) at (-3.267, 1.733) {2};
    \node[black] (33) at (-4.134, 0) {3};
    \node (30) at (-4.134,1) {\(\Gamma_{2}\)};
    \arc{32}{31}
    \arc{33}{32}
    \arc{31}{33}

    \node (40) at (-1.8,1.4) {\(\psi\)};
    \node (41) at (-1.8,1) {\(\longmapsto\)};

    \node[white] (1) at (3, 2.5) {1};
    \node[white] (2) at (3, 1.0) {2};
    \node[white] (3) at (3, 0) {3};

    \node[black] (4) at (0, 1.733) {4};
    \node[black] (5) at (1.733, 1.733) {5};
    \node[black] (6) at (0.866, 0) {6};

    \node[black] (7) at (4.2, 1.733) {7};
    \node[black] (8) at (5.933, 1.733) {8};
    \node[black] (9) at (5.066, 0) {9};

    \node[grey] (10) at (-0.866, 0) {\!10\!};
    \node[grey] (11) at (6.8, 0) {\!11\!};

    \arcDash{4}{1}
    \arcDash{7}{1}
    \arcDash{5}{2}
    \arcDash{8}{2}
    \arcDash{6}{3}
    \arcDash{9}{3}

    \arc{4}{5}
    \arc{5}{6}
    \arc{6}{4}

    \arc{8}{7}
    \arc{9}{8}
    \arc{7}{9}

    \arcDot{4}{10}
    \arcDot{5}{10}
    \arcDot{6}{10}
    \arcDot{7}{11}
    \arcDot{8}{11}
    \arcDot{9}{11}

  \end{tikzpicture}
  \caption{\label{fig-set-of-graphs-to-extended-graph}
    Labelled digraphs \(\Gamma_{1}\) and \(\Gamma_{2}\) on \(\Omega \coloneqq \{1,2,3\}\) with all labels \emph{black},
    and a representative labelled digraph of the extended graph \(\psi(\{\Gamma_{1}, \Gamma_{2}\})\), with \(\Lambda \coloneqq \N\) and \(\psi\) from Section~\ref{sec-perfect-set-of-graphs}.
    Vertices and arcs are differently patterned, according to label.
    The stabiliser in \(\Sym{\{1,\ldots,11\}}\) of this labelled digraph preserves \(\Omega\) setwise, and its restriction to \(\Omega\) is the setwise stabiliser of \(\{\Gamma_{1},\Gamma_{2}\}\) in \(\SOm\).
  }
\end{figure}
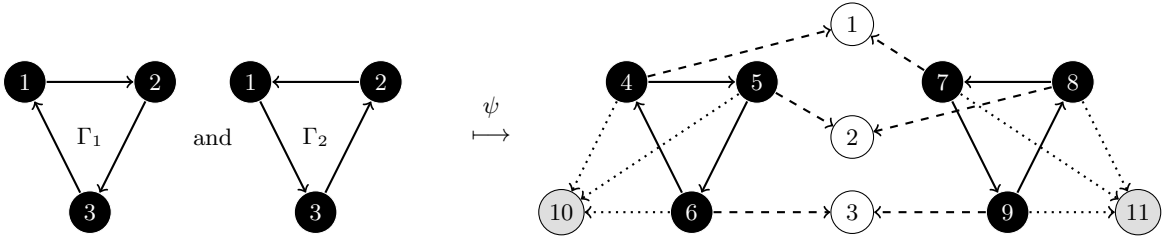

Different choices in the construction lead to labelled digraphs that differ only by a permutation of the vertices in \(\Lambda \setminus \Omega\), so \(\psi\) is well-defined.
The injectivity of \(\psi\) is guaranteed by the vertices in \(X\) and their arcs; these could be omitted for sets of connected labelled digraphs.
It is clear that \(\psi\) is \(\SOm\)-invariant.

\begin{remark}

\mbox{}\\
(a)  The proof of Lemma~\ref{lem-every-group-is-a-set-of-lists-stab} can be adapted to show that when \(|\Omega| > 3\),
  there exist stabilisers and transporter sets of sets of labelled digraphs on \(\Omega\) without perfect refiners in graph backtracking.\\
(b)  The fundamental idea behind this technique can be adapted to build perfect refiners for the stabilisers and transporter sets of extended graphs in extended graph backtracking.
\end{remark}

\subsubsection{Stabilisers and transporters of sets of labelled digraph stacks}\label{sec-perfect-set-of-graph-stacks}

Let \(\zeta\) be the injective \(\SOm\)-invariant function from the set of all labelled digraphs stacks on \(\Omega\) to the set of all labelled digraphs on \(\Omega\) that is given in~\cite[Section~3.1]{GB_published}.
Moreover, let \(k \in \N_{0}\).
Then the function that maps any set \(\{S_{1}, \ldots, S_{k}\}\) of labelled digraph stacks on \(\Omega\) to the set of labelled digraphs \(\{\zeta(S_{1}), \ldots, \zeta(S_{k})\}\) is injective and \(\SOm\)-invariant.
Thus we can use Theorem~\ref{thm-perfect-stabilisers-and-transporters} in combination with Section~\ref{sec-perfect-set-of-graphs} to construct perfect refiners in extended graph backtracking for the stabilisers and transporter sets of arbitrary sets of labelled digraphs stacks on \(\Omega\).
This allows perfect refiners in graph backtracking for the stabilisers and transporter sets of one kind of object to be converted into perfect refiners in extended graph backtracking for the stabilisers and transporter sets of sets of those objects.

\subsubsection{Normalisers and sets of conjugating elements of subgroups}\label{sec-perfect-2-closed-normaliser}

Any subgroup of a group \(K\) acts on the set of subgroups of \(K\) by conjugation and the stabiliser of a subgroup \(G\) under conjugation by a subgroup \(H\) is the \emph{normaliser of \(G\) in \(H\)}, denoted \(N_{H}(G)\).

Computing normalisers and deciding subgroup conjugacy are typical use cases for backtrack search, but existing techniques seem to find these problems particularly difficult.
Describing refiners for normalisers has been and continues to be an important area of research, see for example~\cite{Theissen97,MunSeeThesis}.

One intuitive explaination for the difficulty is that a normaliser may permute structures of a group (such as orbits) that would be fixed in the context of other search problems, which may reduce the potential for pruning.
In the earlier parts of Section~\ref{sec-perfect-graphsplus} and in Table~\ref{table-perfect-stab}, we have seen constructions of perfect refiners in extended graph backtracking for the stabilisers and transporter sets in \(\SOm\) of many kinds of \emph{sets} of objects that do not have perfect refiners in existing backtracking settings.

This suggests that the extended graph backtracking technique may lend itself well to the development of better refiners for normalisers and for sets of conjugating elements of subgroups.
The strategy would be to identify structures that are permuted by the normaliser, and to then develop refiners for the stabilisers and transporter sets of sets of such structures.
As a first step, we examine orbital graphs.

\begin{prop}
  \label{prop-normaliser-permutes-orbitals}
  Let \(G,H \leq \SOm\) and \(x \in \SOm\).
  If \(G^{x} = H\), then \(x\) transports the set of orbital graphs of \(G\) to the set of orbital graphs of \(H\).
  In particular, the normaliser of \(G\) in \(\SOm\) stabilises the set of orbital graphs of \(G\).
\end{prop}

\begin{proof}
  It is routine to verify that if \(G^{x} = H\), then for all \(\alpha,\beta \in \Omega\), \(x\) transports the orbital graph of \(G\) with base-pair \((\alpha,\beta)\) to the orbital graph of \(H\) with base-pair \((\alpha^{x},\beta^{x})\).
\end{proof}

Proposition~\ref{prop-normaliser-permutes-orbitals} implies that the function \(\mu\) that maps a subgroup of \(\SOm\) to its set of orbital graphs is \(\SOm\)-invariant.
In Section~\ref{sec-perfect-set-of-graphs}, we described perfect refiners in extended graph backtracking for the stabilisers and transporter sets of arbitrary sets of digraphs,
and so by Theorem~\ref{thm-perfect-stabilisers-and-transporters}, we can use \(\mu\) to obtain refiners for the normalisers and sets of conjugating elements of subgroups.
However, if \(|\Omega| \geq 4\), then \(\mu\) is not injective (consider different 2-transitive subgroups of \(\SOm\)), so these refiners are not necessarily perfect.
Nevertheless, \(\mu\) gives many instances of perfect refiners.
For example, it is clear that the set of 2-closed subgroups of \(\SOm\) is closed under conjugation by \(\SOm\), and that the restriction of \(\mu\) to this set is injective. Therefore those refiners are perfect.

Furthermore, by Proposition~\ref{prop-normaliser-permutes-orbitals}, the refiner described above for the normaliser of a subgroup \(G \leq \SOm\) is perfect if and only if \(N_{\SOm}(G)\) is equal to the stabiliser in \(\SOm\) of the set of orbital graphs of \(G\).
By Lemma~\ref{lem-orbital-setstab-is-2closure-norm},
these are the subgroups whose normaliser in \(\SOm\) coincides with the normaliser of its 2-closure.
Example~\ref{ex-orbital-setstab} shows that this includes more than just the 2-closed subgroups of \(\SOm\).

\begin{lemma}\label{lem-orbital-setstab-is-2closure-norm}
  Let \(G \leq \SOm\).
  The stabiliser in \(\SOm\) of the set of orbital graphs of \(G\) is the normaliser in \(\SOm\) of the 2-closure of \(G\).
\end{lemma}

\begin{proof}
  Let \(K\) denote the 2-closure of \(G\).
  The orbital graphs of \(G\) and \(K\) coincide by definition, and so \(N_{\SOm}(G)\) is contained in the stabiliser by Proposition~\ref{prop-normaliser-permutes-orbitals}.
  Conversely, let \(x \in \SOm\) stabilise the set of orbital graphs of \(G\) (and \(K\)) and let \(g \in K\).
  For all orbital graphs \(\Gamma\) of \(K\), \(\Gamma^{x^{-1}}\) is an orbital graph of \(K\) by assumption, and \(\Gamma^{x^{-1}g} = \Gamma^{x^{-1}}\) since \(g \in K\). Hence \(\Gamma^{x^{-1} g x} = \Gamma\), and \(x^{-1} g x \in K\).
\end{proof}

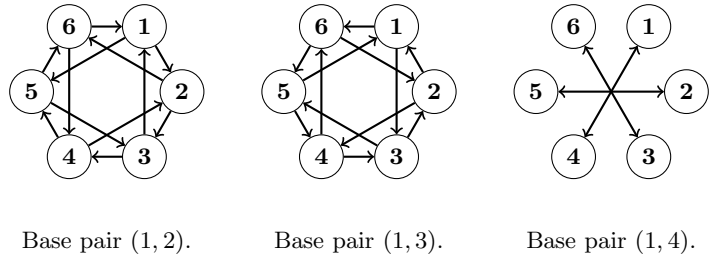
\begin{figure}[!ht]
  \small
  \centering
  \begin{tikzpicture}
    \foreach \x in {1,2,3,4,5,6} {
      \node[circle, draw=black] (\x) at (-\x*60+120:1cm) {\(\mathbf{\x}\)};};
    \foreach \x/\y in {1/2,1/5,2/3,2/6,3/1,3/4,4/2,4/5,5/3,5/6,6/1,6/4} {\arc{\x}{\y}};
    \node at (0, -2) {Base pair \((1,2)\).};
  \end{tikzpicture}
  \qquad
  \begin{tikzpicture}
    \foreach \x in {1,2,3,4,5,6} {
      \node[circle, draw=black] (\x) at (-\x*60+120:1cm) {\(\mathbf{\x}\)};};
    \foreach \x/\y in {1/3,1/6,2/1,2/4,3/2,3/5,4/3,4/6,5/1,5/4,6/2,6/5} {\arc{\x}{\y}};
    \node at (0, -2) {Base pair \((1,3)\).};
  \end{tikzpicture}
  \qquad
  \begin{tikzpicture}
    \foreach \x in {1,2,3,4,5,6} {
      \node[circle, draw=black] (\x) at (-\x*60+120:1cm) {\(\mathbf{\x}\)};};
    \foreach \x/\y in {1/4,2/5,3/6} {\arcSym{\x}{\y}};
    \node at (0, -2) {Base pair \((1,4)\).};
  \end{tikzpicture}
  \caption{\label{fig-orbital-setstab}
  The orbital graphs of the subgroup \(G\) of \(\Sym{\{1,\ldots,6\}}\) from Example~\ref{ex-orbital-setstab}.
  }
\end{figure}

\begin{example}\label{ex-orbital-setstab}
  Let \(\Omega \coloneqq \{1,\ldots,6\}\) and let \(G \coloneqq \< (1\,2\,3)(4\,5\,6), (1\,4)(2\,5) \> \leq \SOm\).
  The orbital graphs of \(G\) are depicted in Figure~\ref{fig-orbital-setstab}.
  The normaliser of \(G\) in \(\SOm\) is the stabiliser in \(\SOm\) of the set of orbital graphs of \(G\), even though \(G\) is not 2-closed. We note that \(G < N_{\SOm}(G) < \SOm\).
\end{example}


\section{Closing remarks}\label{sec-outro}

We have introduced the concept of perfect refiners, which gives a way of comparing the available pruning power in the various backtracking frameworks and which, within a framework, gives a way of comparing refiners for a given set.
This is naturally complemented by the introduction of extended graph backtracking.
We have also discussed the existence of perfect refiners in the different frameworks, and given concrete examples of perfect refiners that are implemented in \textsc{Vole}.

This work suggests several obvious questions and directions for further investigation.

For any given search problem, which backtracking framework is the most appropriate for solving it, and how should this decision be made?  Having decided upon a framework, which refiners should be used?  In which other ways can we compare and understand refiners?

One obvious line of inquiry is the development of further methods for comparing different refiners for the same set, and of measuring the `quality' of a given refiner.
The notion of perfectness is binary, and it fails to capture any nuance in the way that a refiner may fail to be perfect.

It would also be useful to better understand the performance implications of using a given refiner in a backtrack search algorithm.
Roughly speaking, a perfect refiner for a given set is best possible in terms of search size.
On the other hand, like any refiner, a perfect refiner may be impractically expensive to compute with, and this might partially or fully override the search-size advantage.
For refiners in extended graph backtracking, we could begin to understand the interplay between these effects by 
distinguishing refiners according to how many vertices and arcs their extended graphs require.

We have seen in Proposition~\ref{prop-arbitrary-group-graph} and Lemma~\ref{lem-extended-perfect} that every group and coset has a perfect refiner in extended graph backtracking, but that the digraphs in~\cite{Bouwer69,MathOverflow} that underpin these results may be impractical for computation.
Extended graph backtracking would therefore benefit from a better understanding of the theoretical and practical possibilities for representing groups by digraphs of the kind in Proposition~\ref{prop-arbitrary-group-graph}, because then we could work on finding perfect refiners, or decide that that is infeasible.

As mentioned in Section~\ref{sec-perfect-2-closed-normaliser}, extended graph backtracking appears to be well suited to the development of better refiners for use in normaliser and subgroup conjugacy computations. This is already an active area of research, motivated partly by the fact that good algorithms for normaliser search exist (e.g. in \texttt{MAGMA}, see~\cite{Magma}). 
Our work in this direction will be continued in the future.

Finally, we note that there is ongoing work to apply the ideas of the graph backtracking and extended graph backtracking frameworks to canonisation algorithms in finite symmetric groups.


\phantomsection\addcontentsline{toc}{section}{References}
\bibliography{perfect-refiners}{}
\bibliographystyle{alpha}

\end{document}